\newif\ifArXiv 	

\ArXivtrue

\ifArXiv
         \documentclass[review,3p,number]{elsarticle_arXiv}    
         \usepackage{amsthm} 
         \usepackage[font=normalsize]{caption}    
\else
	\documentclass[journal,twoside,web]{ieeecolor}
	\usepackage{generic}
	\usepackage{cite}
\fi

\usepackage{amsmath,amssymb,amsfonts}
\usepackage{algorithmic}
\usepackage{graphicx}
\usepackage{textcomp}
\usepackage{hyperref}
\usepackage{soul}
\usepackage{bm}
\usepackage{balance}

\usepackage{booktabs}

\newcommand{\R}{\mathbb{R}}
\newcommand{\K}{\mathcal{K}}
\newcommand{\E}{\mathbb{E}}
\newcommand{\B }{\mathbb{B}}
\newcommand{\mc}{\mathcal}
\newcommand{\ego}{\mathbb E}
\newcommand{\ob}{\mathbb O}
\newcommand{\dist}{\textnormal{dist}}
\newcommand{\pen}{\textnormal{pen}}
\newcommand{\sd}{\textnormal{sd}}

\ifArXiv
	\theoremstyle{plain}
\fi
\newtheorem{theorem}{Theorem}
\newtheorem{lemma}{Lemma}
\newtheorem{proposition}{Proposition}

\ifArXiv
	\theoremstyle{definition}
\fi

\ifArXiv
	\theoremstyle{remark}
\fi
\newtheorem{remark}{Remark}

\ifArXiv
        
\begin{document}
        \begin{frontmatter}
        \title{Optimization-Based Collision Avoidance}
        \author[ucb]{Xiaojing Zhang}
        \author[ethz]{Alexander Liniger}
        \author[ucb]{Francesco Borrelli}
        \address[ucb]{Model Predictive Control Laboratory, Department of Mechanical Engineering, University of California, Berkeley, USA}
        \address[ethz]{Automatic Control Laboratory, Department of Information Technology and Electrical Engineering, ETH Zurich, Switzerland}
        \tnotetext[fninfo]{E-mail addresses: \{xiaojing.zhang, fborrelli\}@berkeley.edu, liniger@control.ee.ethz.ch}
        \begin{abstract}	
        This paper presents a novel method for reformulating non-differentiable collision avoidance constraints into \emph{smooth} nonlinear constraints using strong duality of convex optimization. We focus on a controlled object whose goal is to avoid obstacles while moving in an $n$-dimensional space. The proposed reformulation does not introduce approximations, and applies to general obstacles and controlled objects that can be represented as the union of convex sets. We connect our results with the notion of signed distance, which is widely used in traditional trajectory generation algorithms. Our method can be applied to generic navigation and trajectory planning tasks, and the smoothness property allows the use of general-purpose gradient- and Hessian-based optimization algorithms. Finally, in case a collision cannot be avoided, our framework allows us to find ``least-intrusive" trajectories, measured in terms of penetration. We demonstrate the efficacy of our framework on a quadcopter navigation and automated parking problem, and our numerical experiments suggest that the proposed methods enable real-time optimization-based trajectory planning problems in tight environments.  Source code of our implementation is provided at \url{https://github.com/XiaojingGeorgeZhang/OBCA}.
        
        \end{abstract}
        
        \end{frontmatter}
        
\else
        \def\BibTeX{{\rm B\kern-.05em{\sc i\kern-.025em b}\kern-.08em
            T\kern-.1667em\lower.7ex\hbox{E}\kern-.125emX}}
        \markboth{\journalname, VOL. XX, NO. XX, XXXX 2018}
        {Zhang \MakeLowercase{\textit{et al.}}: Optimization-based Collision Avoidance}
        \begin{document}
        \title{Optimization-based Collision Avoidance}
        \author{Xiaojing Zhang,  Alexander Liniger,  and Francesco Borrelli
        \thanks{This research has been partially funded by the Hyundai Center of Excellence at UC Berkeley.}
        \thanks{X.\ Zhang and F.\ Borrelli are with the Model Predictive Control Laboratory, Department of Mechanical Engineering, University of California at Berkeley, CA 94720, USA (email: xiaojing.zhang@berkeley.edu, fborrelli@berkeley.edu). }
        \thanks{A.\ Liniger is with the Automatic Control Laboratory, Department of Information Technology and Electrical Engineering, ETH Zurich, 8092 Zurich, Switzerland (liniger@control.ee.ethz.ch).}
        }

	\maketitle

        \begin{abstract}
        This paper presents a novel method for reformulating non-differentiable collision avoidance constraints into \emph{smooth} nonlinear constraints using strong duality of convex optimization. We focus on a controlled object whose goal is to avoid obstacles while moving in an $n$-dimensional space. The proposed reformulation does not introduce approximations, and applies to general obstacles and controlled objects that can be represented as the union of convex sets. We connect our results with the notion of signed distance, which is widely used in traditional trajectory generation algorithms. Our method can be applied to generic navigation and trajectory planning tasks, and the smoothness property allows the use of general-purpose gradient- and Hessian-based optimization algorithms. Finally, in case a collision cannot be avoided, our framework allows us to find ``least-intrusive" trajectories, measured in terms of penetration. We demonstrate the efficacy of our framework on a quadcopter navigation and automated parking problem, and our numerical experiments suggest that the proposed methods enable real-time optimization-based trajectory planning problems in tight environments.  Source code of our implementation is provided at \url{https://github.com/XiaojingGeorgeZhang/OBCA}.
        \end{abstract}
        
\fi

\section{Introduction}
\label{sec:introduction}
\ifArXiv
       Maneuvering  
\else
	\IEEEPARstart{M}{aneuvering} 
\fi
autonomous systems in an environment with obstacles is a challenging problem that arises in a number of practical applications including robotic manipulators and trajectory planning for autonomous systems such as self-driving cars and quadcopters. In almost all of those applications, a fundamental feature is the system's ability to avoid collision with obstacles which are, for example, humans operating in the same area, other autonomous systems, or static objects such as walls.

Optimization-based trajectory planning algorithms such as Model Predictive Control (MPC) have received significant attention recently, ranging from (unmanned) aircraft to robots to autonomous cars \cite{RichardsHow_ACC2002, BorrelliKeviczkyBalas_CDC2004, Diehl2006, BlackmoreWilliamsManipularObstacleAvoid2006, gao2010predictive, GoulartTrajectory2011, BlackmoreOnoWilliams2011, Salmah2013,  SchulmanAbeelColisionChecking2014, Liniger_OCAM2015, rosolia2017autonomousrace, FunkeGerdes2017}. 
This can be attributed to the increase in computational resources, the availability of robust numerical algorithms for solving optimization problems, as well as MPC's ability to systematically encode system dynamics and constraints inside its formulation. 

One fundamental challenge in optimization-based trajectory planning is the appropriate formulation of collision avoidance constraints, which are known to be non-convex and computationally difficult to handle in general. While a number of formulations have been proposed in the literature for dealing with collision avoidance constraints, they are typically limited by one of the following features: $(i)$ The collision avoidance constraints are approximated through linear constraint, and it is difficult to establish the approximation error \cite{SchulmanAbeelColisionChecking2014}; $(ii)$ Existing formulations focus on point-mass controlled objects, and are not applicable to full-dimensional objects; $(iii)$ When the obstacles are polyhedral, then the collision avoidance constraints are often reformulated using integer variables \cite{GrossmannReviewDisjunctive2002}. 
While this reformulation is attractive for linear systems with convex constraints since in this case a mixed-integer convex optimization problem can be solved, integer variables should generally be avoided when dealing with nonlinear systems when designing real-time controllers for robotic systems.

In this paper, we focus on a \emph{controlled object} that moves in a general $n$-dimensional space while avoiding obstacles, and propose a novel approach for modeling obstacle avoidance constraints that overcomes the aforementioned limitations. Specifically, the contributions of this paper can be summarized as follows:
\begin{itemize}
    \item We show that if the  controlled object and the obstacles are described by convex sets such as polytopes or ellipsoids (or can be decomposed into a finite union of such convex sets), then the collision avoidance constraints can be exactly and non-conservatively reformulated as a set of smooth non-convex constraints. This is achieved by appropriately reformulating the \emph{distance}-function between two convex sets using strong duality of convex optimization. 
   
    \item We provide a second formulation for collision avoidance based on the notion of \emph{signed distance}, which  characterizes not only the distance between two objects but also their penetration. This reformulation allows us to compute ``least-intrusive" trajectories in case collisions cannot be avoided. 
        
    \item We demonstrate the efficacy of the proposed obstacle avoidance reformulations on a quadcopter trajectory planning problem and autonomous parking application, where the controlled vehicles must navigate in tight environments. We show that both the distance reformulation and the signed distance reformulation enable real-time path planning and find trajectories even in challenging circumstances.
\end{itemize}
Furthermore, since both our formulations allow the  incorporation of system dynamics and input constraints, the generated trajectories are \emph{kinodynamically feasible}, and hence can be tracked by simple low-level controllers.

This paper is organized as follows: Section \ref{sec:probDescription} introduces the problem setup. Section~\ref{sec:PointMass} presents the collision avoidance and minimum-penetration formulations for the case when the controlled object is a point mass. These results are then extended to full-dimensional controlled objects in Section~\ref{sec:FullDimObj}. Numerical experiments demonstrating the efficacy of the proposed method are given in Sections~\ref{sec:ex_quadcopter} and \ref{sec:ex_autParking}, and conclusions are drawn in Section~\ref{sec:conclusion}. The Appendix contains auxiliary results needed to prove the main results of the paper. 
The source code of a quadcopter navigation example and autonomous parking example described in Sections~\ref{sec:ex_quadcopter} and \ref{sec:ex_autParking} is provided at \url{https://github.com/XiaojingGeorgeZhang/OBCA}.

\subsection*{Related Work}
A large body of work exists on the topic of obstacle avoidance. In this paper, we do not review, or compare, optimization-based collision avoidance methods with alternative approaches such as those based on dynamic programming \cite{SchildbachDPParking2016}, reachability analysis \cite{TomlinPappasSastry1998, MargellosLygerosTAC2011,MargellosLygerosTCST2013},  graph search \cite{LikhachevFergusonPlanningLongDyn2009, DolgovThrunPathPlanning2010, LikhachevMultiAstar_IJRR2016}, (random) sampling \cite{LaValle2001, Ziegler2008, KaramanFrazzoli2011, BanzhafHybridCurvature2017}, or interpolating curves \cite{Vorobieva2014}. Indeed, collision avoidance problems are known to be NP-hard in general \cite{canny1988complexity}, and all practical methods constitute some sort of ``heuristics", whose performance depends on the specific problem and configuration at hand. 
In the following, we briefly review optimization-based approaches, and refer the interested reader to \cite{lavalle2006planning,Campbell2010, Katrakazas2015416, PadenFrazolliSurveyMotionPlanning2016, GonzalezReviewMotionPlanning2016} for a comprehensive review on existing trajectory planning and obstacle avoidance algorithms.

The basic idea in optimization-based methods is to express the collision avoidance problem as an optimal control problem, and then solve it using numerical optimization techniques. 
One way of dealing with obstacle avoidance is to use unconstrained optimization, in which case the objective function is augmented with ``artificial potential fields" that represent the obstacles \cite{Khatib1986, KalakrishnanSTOMP_2011, ParkITOMP_2012, ZuckerDraganCHOMP_2013, Gennert2016_BiRRTOpt}.
More recently, methods based on constrained optimization has attracted attention in the control community, due to its ability to explicitly formulate collision avoidance through constraints \cite{RichardsHow_ACC2002, BlackmoreWilliamsManipularObstacleAvoid2006, GoulartTrajectory2011, SchulmanAbeelColisionChecking2014}.
Broadly speaking, constrained optimization-based collision-avoidance algorithms can be divided into two cases based on the modeling of the controlled object: point-mass models and full-dimensional objects. Due to its conceptual simplicity, the vast majority of literature focuses on collision avoidance for point-mass models, and consider the shape of the controlled object by inflating the obstacles. The obstacles are generally assumed to be either polytopes or ellipsoids. For polyhedral obstacles, disjunctive programming can be used to ensure  collision avoidance, which is often reformulated as a mixed-integer optimization problem \cite{RichardsHow_ACC2002, BlackmoreWilliamsManipularObstacleAvoid2006, GrossmannReviewDisjunctive2002}. In case of ellipsoidal obstacles, the collision avoidance constraints can be formulated as a smooth non-convex constraint \cite{Rosolia2017,nageli2017real}, and the resulting optimization problem can be solved using generic non-linear programming solvers.

The case of full-dimensional controlled objects has, to be best of the authors' knowledge, not been widely studied in the context of optimization-based methods, with the exception of \cite{SchulmanAbeelColisionChecking2014, LiShao2015}. In \cite{LiShao2015}, the authors model the controlled object through its vertices and, under the assumption that all involved object are rectangles, ensure collision avoidance by keeping all vertices of the controlled object outside the obstacle. A more general way of handling collision avoidance for full-dimensional controlled object has been proposed in \cite{SchulmanAbeelColisionChecking2014} using the notion of signed distance, where the authors also propose a sequential linearization technique to deal with the non-convexity of the signed distance function.

The approach most closely related to our formulation is probably the work of \cite{GoulartTrajectory2011}, where the authors propose a smooth reformulation of the collision avoidance constraint for point-mass controlled objects and polyhedral obstacles. However, our approach differs from \cite{GoulartTrajectory2011} as $(i)$ our approach  generalizes to full-dimensional controlled objects,  and $(ii)$ we are, based on the notion of penetration, also able to compute least-intrusive trajectories in case collisions cannot be avoided. 

\subsection*{Notation}\label{sec:notation}
Given a proper cone $\mc K \subset\R^l$ and two vectors $a,b\in\R^l$, then $a\preceq_{\mc K} b$ is equivalent to $(b-a)\in\mc K$. If $\K = \R_+^l$ is the standard cone, then $\preceq_{\R_+^l}$ is equivalent to the standard (element-wise) inequality $\leq$. Moreover, $\|\cdot\|_*$ is the dual norm of $\|\cdot\|$, and $\K^* \subset\R^l$ is the dual cone of $\K$. 
The ``space" occupied by the controlled object (e.g., a drone, vehicle, or robot in general) is denoted as $\ego\subset\R^n$; similarly, the space occupied by the obstacles is denoted as $\ob\subset\R^n$.

\section{Problem Description}\label{sec:probDescription}
\subsection{Dynamics, Objective and Constraints}
We assume that the dynamics of the controlled object takes the form 
\begin{equation}\label{eq:stateDynamics}
    x_{k+1} = f(x_k, u_k),
\end{equation}
where $x_k\in\R^{n_x}$ is the state of the controlled object at time step $k$ given an initial state $x_0=x_S$, $u_k\in\R^{n_u}$ is the control input, and $f\colon\R^{n_x} \times \R^{n_u} \to \R^{n_x}$  describes the dynamics of the system. 
In most cases, the state $x_k$ contains information such as the position $p_k\in\R^n$ and angles $\theta_k\in\R^n$ of the controlled object, as well the velocities $\dot p_k$ and angular rates $\dot \theta_k$.
In this paper, we assume that no disturbance is present, and that the system is subject to input and state constraints of the form 
\begin{equation}\label{eq:stateInputConstraints}
	h(x_k, u_k )\leq0,
\end{equation}
where  $h:\R^{n_x} \times \R^{n_u} \to \R^{n_h}$, $n_h$ is the number of constraints, and the inequality in \eqref{eq:stateInputConstraints} is interpreted element-wise. Our goal is to find a control sequence, over a horizon $N$, which allows the controlled object to navigate from the initial state $x_S$ to its final state $x_F\in\R^{n_x}$, while optimizing some objective function $J = \sum_{k=0}^N \ell(x_k, u_k)$, where $\ell\colon \R^{n_x} \times \R^{n_u} \to \R$  is a stage cost, and avoiding $M\geq1$ obstacles $\ob^{(1)},\ob^{(2)},\ldots,\ob^{(M)} \subset \R^n$.
Throughout this paper, we assume that the functions $f(\cdot,\cdot)$, $h(\cdot,\cdot)$ and $\ell(\cdot,\cdot)$ are smooth.
Smoothness is assumed for simplicity, although all forthcoming statements apply equally to cases when those functions are twice continuously differentiable.

\subsection{Obstacle and Controlled Object Modeling}
Given the state $x_k$, we denote by $\ego(x_k)\subset\R^n$ the ``space" occupied by the controlled object at time $k$, which we assume is a subset of $\R^n$. The collision avoidance constraint at time $k$ is now given by\footnote{In this paper, we only consider collision avoidance constraints that are associated with the position and geometric shape of the controlled object, which are typically defined by its position $p_k$ and angles $\theta_k$. This is not a restriction of the theory as the forthcoming approaches can be easily generalized to collision avoidance involving other states, but done to simplify exposition of the material.}
\begin{equation}\label{eq:collAvoid}
	\ego(x_k) \cap \ob^{(m)} = \emptyset, \quad \forall m = 1 ,\ldots,M.  
\end{equation}
Constraint \eqref{eq:collAvoid} is non-differentiable in general, e.g., when the obstacles are polytopic \cite{GoulartTrajectory2011, SchulmanAbeelColisionChecking2014}. In this paper, we will remodel \eqref{eq:collAvoid} in such a way that both continuity and differentiability are preserved.
To this end, we assume that  the obstacles $\ob^{(m)}$ are convex compact sets with non-empty relative interior\footnote{Non-convex obstacles can often be approximated/decomposed as the union of convex obstacles}, and can be represented as
\begin{equation}\label{eq:convOb}
	\ob^{(m)} = \{y\in\R^n \colon A^{(m)} y \preceq_{\K} b^{(m)}\},
\end{equation}
where  $A^{(m)}\in\R^{l \times n}$, $b^{(m)}\in\R^l$, and $\K\subset\R^l$ is a closed convex pointed cone with non-empty interior. Representation \eqref{eq:convOb} is entirely generic since any compact convex set admits a conic representation of the form \eqref{eq:convOb} \cite[p.15]{rockafellar1970}. In particular, polyhedral obstacles can be represented as  \eqref{eq:convOb} by choosing $\K = \R_+^l$; in this case $\preceq_{\K}$ corresponds to the well-known element-wise inequality $\leq$. Likewise, ellipsoidal obstacles can be represented by letting $\K$ be the second-order cone, see \cite{boyd2004convex} for details.
To simplify the upcoming exposition, the same cone $\K$ is assumed for all obstacles; the extension to obstacle-specific cones $\K^{(m)}$ is straight-forward. 

In this paper, we will consider controlled objects $\E(x_k)$ that are modeled as \emph{point-masses} as well as \emph{full-dimensional} objects. In the former case, $\E(x_k)$ simply extracts the position $p_k$ from the state $x_k$, i.e.,
\begin{subequations}
\begin{equation}\label{eq:transfEgoPointMass}
	\ego(x_k) = p_k.
\end{equation} 
In the latter case, we will model the controlled object $\E(x_k)$ as the rotation and translation of an ``initial" convex set $\B \subset\R^n$, i.e.,
\begin{equation}\label{eq:transfEgoFullSet}
	\ego(x_k) = R(x_k) \B + t(x_k), \quad \B := \{y \colon G y \preceq_{\bar \K} g\},
\end{equation}
where $R\colon \R^{n_x} \to \R^{n \times n}$ is an (orthogonal) rotation matrix and $t\colon \R^{n_x} \to \R^n$ is the translation vector. The matrices $(G,g) \in \R^{h\times n} \times \R^h$ and the cone $\bar\K \subset\R^h$, which we assume is closed, convex and pointed, define the shape of our initial (compact) set $\B$ and are assumed known. Often, the rotation matrix $R(\cdot)$ depends on the angles $\theta_k$ of the controlled object, while the translation vector $t(\cdot)$ depends on the position $p_k$ of the controlled object.
We assume throughout that the functions $R(\cdot)$ and $t(\cdot)$ are smooth.
\end{subequations}

\subsection{Optimal Control Problem with Collision Avoidance}

By combining \eqref{eq:stateDynamics}--\eqref{eq:collAvoid}, the constrained finite-horizon optimal control problem with collision avoidance constraint is given by
\begin{equation}\label{eq:MPC_genCollAvoid}
	\begin{array}{lll}
		\displaystyle\min_{\mathbf x, \mathbf u}  & \displaystyle\sum_{k=0}^{N} \ell(x_k, u_k) \\
		\ 	\text{s.t.} & x_0 = x_S,~ x_{N+1} = x_F, \\
					& \!\!\!\!\!\!\!\!\!\!  \left. \begin{array}{ll}
					&x_{k+1} = f(x_k,u_k), \\
					& h(x_k, u_k) \leq 0, \\
					& \ego(x_k) \cap \ob^{(m)}= \emptyset, \\
					\end{array} \right\} 
					\begin{array}{lll} 
						k=0,\ldots,N,\\[0ex]
						m=1,\ldots,M,
					\end{array}
	\end{array}
\end{equation}
where $\ego(x_k)$ is either given by \eqref{eq:transfEgoPointMass} (point-mass model) or \eqref{eq:transfEgoFullSet} (full-dimensional set), $\mathbf x := [x_0,x_1,\ldots,x_{N+1}]$ is the collection of all states, and $\mathbf u := [u_0,u_1,\ldots,u_{N}]$ is the collection of all inputs. 
A key difficulty in solving problem \eqref{eq:MPC_genCollAvoid}, even for linear systems with convex objective function and convex state/input constraints, is the presence of the collision-avoidance constraints $\ego(x_k)\cap\ob^{(m)}=\emptyset$, which in general are non-convex and non-differentiable \cite{GoulartTrajectory2011,SchulmanAbeelColisionChecking2014}. In the following, we present two novel approaches for modeling collision avoidance constraints  that preserve continuity and differentiability, and are amendable for use with existing off-the-shelf gradient- and Hessian-based optimization algorithms.

\subsection{Collision Avoidance}
A popular way of formulating collision avoidance is based on the notion of \emph{signed distance} \cite{SchulmanAbeelColisionChecking2014}
\begin{equation}\label{eq:signedDist}
	\text{sd}(\E(x),\ob) := \text{dist}(\E(x),\ob) - \pen(\E(x),\ob), 
\end{equation}
where $\dist(\cdot,\cdot)$ and $\pen(\cdot,\cdot)$ are the distance and penetration function, and  are defined as
\begin{subequations}
\begin{align}
	\text{dist}(\E(x),\ob) &:= \min_{t}\{\|t\| \colon \left(\E(x)+t \right) \cap \ob \neq \emptyset \}, \label{eq:defDist}\\
	\text{pen}(\E(x),\ob) &:= \min_{t}\{\|t\| \colon \left(\E(x)+t \right) \cap \ob = \emptyset \}\label{eq:defPen}.
\end{align}
\end{subequations}
Roughly speaking, the signed distance is positive if $\ego(x)$ and $\ob$ do not intersect, and negative if they overlap.
Therefore, collision avoidance can be ensured by requiring $\text{sd}(\E(x),\ob)>0$. Unfortunately, directly enforcing $\text{sd}(\E(x),\ob)>0$ inside the optimization problem \eqref{eq:MPC_genCollAvoid} is generally difficult since it is non-convex and non-differentiable in general \cite{SchulmanAbeelColisionChecking2014}. Furthermore, for optimization algorithms to be numerically efficient, they require an explicit representation of the functions they are dealing with, in this case $\text{sd}(\cdot,\cdot)$. This, however, is difficult to obtain in practice since $\text{sd}(\cdot,\cdot)$ itself is the solution of the  optimization problems \eqref{eq:defDist} and \eqref{eq:defPen}.
As a result, existing algorithms approximate \eqref{eq:signedDist} through local linearization \cite{SchulmanAbeelColisionChecking2014}, for which it is difficult to establish bounds on approximation errors. 

In the following, we propose two reformulation techniques for obstacles avoidance that overcome the issues of non-differentiability and that do not require an explicit representation of the signed distance.  We begin with point-mass models in Section~\ref{sec:PointMass}, and treat the general case of full-dimensional controlled objects in Section~\ref{sec:FullDimObj}.

\section{Collision Avoidance for Point-Mass Models}\label{sec:PointMass}
In this section, we first present a smooth reformulation of \eqref{eq:MPC_genCollAvoid} when $\ego(x_k) = p_k$ in Section~\ref{sec:collFree_pointMass}, and then extend the approach in Section~\ref{subsec:min-PenTraj_pointMass} to generate minimum-penetration trajectories in case collisions cannot be avoided. To simplify notation, the time indices $k$ are omitted in the remainder of this section.

\subsection{Collision-Free Trajectory Generation}\label{sec:collFree_pointMass}
\begin{proposition}\label{prop:distDual_pointMass}
	Assume that the obstacle $\ob$ and the controlled object are given as in \eqref{eq:convOb} and \eqref{eq:transfEgoPointMass}, respectively, and let $\mathsf{d}_\textnormal{min} \geq0 $ be a desired safety margin between the controlled object and the obstacle. Then we have:
	\begin{align}\label{eq:distDual_pointMass}
	&\textnormal{dist}(\ego(x), \ob) > \mathsf{d}_\textnormal{min}  ~\\
	  & \qquad \Longleftrightarrow~  \exists \lambda\succeq_{\K^*}0 \colon (A\,p - b)^\top \lambda  > \mathsf{d}_\textnormal{min},~\|A^\top\lambda\|_* \leq 1. \nonumber
	\end{align}
\end{proposition}

\begin{proof}
 	It follows from \eqref{eq:convOb} and \eqref{eq:defDist} that $\dist(\ego(x),\ob)=\min_{t}\{ \| t \| \colon A(\ego(x) + t) \preceq_\K b\}$. Following \cite[p.401]{boyd2004convex}, its dual problem is given by $\max_{\lambda}\left\{ (A \ego(x) -b)^\top \lambda \colon  \|A^\top \lambda\|_* \leq1,~\lambda \succeq_{\K^*} 0  \right\}$, where $\|\cdot\|_*$ is the dual norm associated to $\|\cdot\|$ and $\K^*$ is the dual cone of $\K$. Since $\ob$ is assumed to have non-empty relative interior, strong duality holds, and $\dist(\ego(x),\ob) = \max_{\lambda}\left\{ (A\, \ego(x) -b)^\top \lambda \colon  \|A^\top \lambda\|_* \leq1,~\lambda \succeq_{\K^*} 0  \right\}$. Hence, for any non-negative scalar $\mathsf{d}_\textnormal{min}$, $\dist(\ego(x),\ob) > \mathsf{d}_\text{min}$ is satisfied if, and only if, there exists $\lambda\succeq_{\K^*}0 \colon (A \, \ego(x) -b)^\top \lambda > \mathsf{d}_\text{min},~ \|A^\top \lambda\|_* \leq1$. The desired result follows from identity \eqref{eq:transfEgoPointMass}. 
\end{proof}

Intuitively speaking, any  variable $\lambda$ satisfying the right-hand-side of \eqref{eq:distDual_pointMass} is a certificate verifying the condition $\dist(\ego(x),\ob) >\mathsf{d}_\text{min}$. Since $\ego(x)\cap \ob = \emptyset$ is equivalent to $\dist(\ego(x),\ob)>0$, the optimal control problem \eqref{eq:MPC_genCollAvoid} for the point-mass model \eqref{eq:transfEgoPointMass} is given by
\begin{equation}\label{eq:MPC_dualDist_pointMass}
	\begin{array}{lll}
		\displaystyle\min_{\mathbf x, \mathbf u,\bm\lambda}  &\displaystyle \sum_{k=0}^{N}  \ell(x_k, u_k)  \\
			\ \text{s.t.} & \! x_0 = x_S,~ x_{N+1} = x_F, \\
					& x_{k+1} = f(x_k,u_k),~h(x_k, u_k) \leq 0, \\
					&(A^{(m)}\, p_k - b^{(m)})^\top \lambda_k^{(m)}  > 0,~\\
					& \|{A^{(m)}}^\top\lambda_k^{(m)}\|_* \leq 1,~ \lambda_k^{(m)}\succeq_{\K^*}0  , \\
					& \textnormal{for  }   k=0,\ldots,N,\ m=1,\ldots,M, 
					
	\end{array}
\end{equation}
where $p_k$ is the position of the controlled object at time $k$, $\lambda_k^{(m)}$ is the dual variable associated with obstacle $\ob^{(m)}$ at time step $k$, and the optimization is performed over the states $\mathbf x$, the inputs $\mathbf u$ and the dual variables $\bm\lambda = [\lambda_0^{(1)},\ldots,\lambda_0^{(m)},\lambda_1^{(1)},\ldots,\lambda_N^{(m)}]$. 
We emphasize that \eqref{eq:MPC_dualDist_pointMass} is an \emph{exact} reformulation of \eqref{eq:MPC_genCollAvoid} and that the optimal trajectory $\mathbf x^* = [x_0^*,x_1^*,\ldots,x_N^*]$ obtained by solving \eqref{eq:MPC_genCollAvoid} is \emph{kinodynamically feasible}.

\begin{remark}\label{rem:smoothness}
    Without further assumptions on the norm $\|\cdot\|$ and the cone $\K$, the last two constraints in \eqref{eq:MPC_dualDist_pointMass} are not guaranteed to be smooth, a property that many general-purpose non-linear optimization algorithms require\footnote{Strictly speaking, these solvers often require the cost function and constraints to be twice continuously differentiable only. Smoothness is assumed in this paper for the sake of simplicity.}. Fortunately, it turns out that these constraints are smooth for the practically relevant cases of $\|\cdot\|$ being the Euclidean distance and $\mathcal K$ either the standard cone or the second-order cone, which allows us to model polyhedral and ellipsoidal obstacles. In these cases, and under the assumption that the functions $f(\cdot,\cdot)$, $h(\cdot,\cdot)$ and $\ell(\cdot,\cdot)$ are smooth, \eqref{eq:MPC_dualDist_pointMass} is a smooth nonlinear optimization problem that is amendable to general-purpose non-linear optimization algorithms such as IPOPT \cite{WachterIPOPT2006}. Without going into details, we point out that smoothness is retained when $\|\cdot\| = \|\cdot\|_p$ is a general $p$-norm, with $p\in(1,\infty)$, and  $\K$ is the cartesian product of $p$-order cones $\K_p := \{(s,z)\colon \|z\|_p \leq s \}$, with $p\in(1,\infty)$. In this case, the dual norm is given by $\|\cdot\|_*=\|\cdot\|_q$ and the dual cone is $(\K_p)^* = \K_q$, where $q$ satisfies $1/p + 1/q = 1$, see \cite{boyd2004convex} for details on dual norms and dual cones.
\end{remark}

While reformulation \eqref{eq:MPC_dualDist_pointMass} can be used for obstacle avoidance, it is limited to finding collision-free trajectories. Indeed, in case collisions cannot be avoided,  the above formulation is not able to find ``least-intrusive" trajectories by softening the constraints. Intuitively speaking, this is because  \eqref{eq:MPC_dualDist_pointMass} is based on the notion of distance, and the distance between two overlapping objects (as is in the case of collision), is always zero, regardless of the penetration. 
From a practical point of view, this implies that slack variables cannot be included in the constraints of \eqref{eq:distDual_pointMass}, because the optimal control problem is not able to distinguish between ``severe" and ``less severe" colliding trajectories. Furthermore, in practice, it is often desirable to soften constraints and include slack variables to ensure feasibility of the (non-convex) optimization problem, since (local) infeasibilities in non-convex optimization problem are known to cause  numerical difficulties. In the following, we show how the above limitations can be overcome by considering the notion of penetration and softening the collision avoidance constraints.

\subsection{Minimum-Penetration Trajectory Generation}\label{subsec:min-PenTraj_pointMass}
In this section, we consider the design of  \emph{minimum-penetration} trajectories for cases when collision cannot be avoided and the goal is to find a ``least-intrusive" trajectory. Following the literature  \cite{CameronCulleyTranslationalDistance1986, Dobkin1993}, we measure ``intrusion" in terms of \emph{penetration} as defined in \eqref{eq:defPen}. 
\begin{proposition}\label{prop:penDual_pointMass}
	Assume that the obstacle $\ob$ and controlled object are given as in \eqref{eq:convOb} and \eqref{eq:transfEgoPointMass}, respectively, and let $\mathsf{p}_\textnormal{max} \geq0 $ be a desired maximum penetration of the controlled object and the obstacle. Then we have:
	\begin{align}\label{eq:penDual_pointMass}
	& \textnormal{pen}(\ego(x), \ob) < \mathsf{p}_\textnormal{max}   \\
	&\qquad \Longleftrightarrow~  \exists \lambda\succeq_{\K^*}0 \colon (b-A\,p )^\top \lambda  < \mathsf{p}_\textnormal{max},~\|A^\top\lambda\|_* = 1. \nonumber
	\end{align}
\end{proposition}
\begin{proof}
	The proof, along with auxiliary lemmas, is given in the Appendix.
\end{proof}
Proposition~\ref{prop:penDual_pointMass} resembles Proposition~\ref{prop:distDual_pointMass} with the  difference that the convex inequality constraint $\|A^\top \lambda\|_* \leq 1$ is replaced with the non-convex equality constraint $\|A^\top \lambda\|_* =1$. In the following, we will see that Propositions~\ref{prop:distDual_pointMass} and \ref{prop:penDual_pointMass} can be combined to represent the \emph{signed distance} as defined in \eqref{eq:signedDist}.
\begin{theorem}\label{thm:signDist_pointMass}
	Assume that the obstacle $\ob$ and the controlled object are given as in \eqref{eq:convOb} and \eqref{eq:transfEgoPointMass}, respectively. Then, for any $\mathsf{d}\in\R$, we have:
	\begin{align}\label{eq:signDist_pointMass}
	& \textnormal{sd}(\ego(x), \ob) > \mathsf{d}  ~\\
	& \qquad \Longleftrightarrow~  \exists \lambda\succeq_{\K^*}0 \colon (A\,p - b)^\top \lambda  > \mathsf{d},~\|A^\top\lambda\|_* = 1. \nonumber
	\end{align}
\end{theorem}
\begin{proof}
	By definition, $\sd(\ego(x),\ob) = \dist(\ego(x),\ob)$ if $\ego(x) \cap \ob = \emptyset$, and $\sd(\ego(x),\ob) = -\pen(\ego(x),\ob)$ if $\ego(x) \cap \ob \neq \emptyset$. Let $\ego(x) \cap \ob \neq \emptyset$, in which case \eqref{eq:signDist_pointMass} follows directly from \eqref{eq:penDual_pointMass}. If $\ego(x) \cap \ob = \emptyset$, then we have from \eqref{eq:distDual_pointMass} that $\sd(\ego(x),\ob) > \mathsf{d}$ is equivalent to $\exists \lambda\succeq_{\K^*}0 \colon (A\,p - b)^\top \lambda  > \mathsf{d},~\|A^\top\lambda\|_* \leq 1$. Due to homogeneity with respect to $\lambda$, if the previous condition is satisfied, then there always exists a (scaled) dual multiplier $\lambda'\succeq_{\K^*}0$ such that $(A\,p - b)^\top \lambda'  > \mathsf{d},~\|A^\top\lambda'\|_* = 1$. This concludes the proof.
\end{proof}
Reformulation \eqref{eq:signDist_pointMass} is similar to reformulation \eqref{eq:distDual_pointMass}, with the difference that \eqref{eq:signDist_pointMass} holds for all $\mathsf{d}\in\R$, while \eqref{eq:distDual_pointMass} only holds for $\mathsf d\geq0$. The ``price" we pay for this generalization is that the convex constraint $\|A^\top \lambda\|_* \leq 1$ is turned into the non-convex equality constraint $\|A^\top \lambda\|_* = 1$ which, as we will see later on, generally results in  longer computation times. Nevertheless, Theorem~\ref{thm:signDist_pointMass} allows us to compute trajectories of least penetration whenever collision cannot be avoided by solving the following soft-constrained \emph{minimum-penetration} problem:
\begin{equation}\label{eq:MPC_genCollAvoid_pointMass}
	\begin{array}{lll}
		\displaystyle\min_{\mathbf x, \mathbf u,\mathbf s, \bm\lambda}  &\displaystyle \sum_{k=0}^{N} \left[ \ell(x_k, u_k) + \kappa\cdot\sum_{m=1}^M s_{k}^{(m)}  \right]  \\
			\ \ \text{s.t.} & x_0 = x(0),~ x_{N+1} = x_F, \\
										&x_{k+1} = f(x_k,u_k),~h(x_k, u_k) \leq 0, \\
					&(A^{(m)}\,p_k - b^{(m)})^\top \lambda_k^{(m)}  > -s_{k}^{(m)},~\\
					& \|{A^{(m)}}^\top\lambda_k^{(m)}\|_* = 1,~ \\ &  s_k^{(m)}\geq0,~\lambda_k^{(m)}\succeq_{\K^*}0, \\  
					& \text{for } k=0,\ldots,N,\ m=1,\ldots,M,
	\end{array}
\end{equation}
where $p_k$ is the position of the controlled object at time $k$, $s_k^{(m)} \in \R_+$ is the slack variable associated to the object $\ob^{(m)}$ at time step $k$, and $\kappa\geq0$ is a weight factor that keeps the slack variable as close to zero as possible.
Without going into details, we point out that the weight $\kappa$ should be chosen ``big enough" such that the slack variables only become active when the original problem is infeasible, i.e., when obstacle avoidance is not possible \cite{borrelli2017predictive}. 
Notice that a positive slack variable implies a colliding trajectory, where the penetration depth is given by $s_k^{(m)}$.
We close this section by pointing out that if, a priori, it is known that a collision-free trajectory can be generated, then formulation \eqref{eq:MPC_dualDist_pointMass} should be given preference over formulation \eqref{eq:MPC_genCollAvoid_pointMass} because the former has fewer decision variables, and because the constraint $\|{A^{(m)}}^\top \lambda_k^{(m)}\|_* \leq 1$ is convex, which generally leads to improved solution times. 
Smoothness of \eqref{eq:MPC_genCollAvoid_pointMass} is ensured if $\|\cdot\|$ is the Euclidean distance, and $\K$ is either the standard cone or the second-order cone, see
Remark~\ref{rem:smoothness} for details.

\section{Collision Avoidance for Full-Dimension Controlled Objects}\label{sec:FullDimObj}
The previous section provided a framework for computing collision-free and minimum-penetration trajectories for controlled objects that are described by point-mass models.
While such models can be used to generate trajectories for ``ball-shaped" controlled objects, done by setting the minimum distance $d_{\min}$ equal to the radius of the controlled object (see Section~\ref{sec:ex_quadcopter} for such an example), it can be restrictive in other cases. 
For example, modeling a car in a parking lot as a Euclidean ball can be very conservative, and  prevent the car from finding a parking spot. To alleviate this issue, we show in this section how the results of Section~\ref{sec:PointMass} can be extended to full-dimensional controlled objects.

\subsection{Collision-Free Trajectory Generation}
Similar to Section~\ref{sec:PointMass}, we begin by first reformulating the distance function, which will allow us to generate collision-free trajectories:
\begin{proposition}\label{prop:distDual_fullDimSet}
	Assume that the controlled object and the obstacle are given as in \eqref{eq:transfEgoFullSet} and \eqref{eq:convOb}, respectively, and let $\mathsf{d}_\textnormal{min} \geq0 $ be a desired safety margin. Then we have:
	\ifArXiv
			\begin{align}\label{eq:distDual_fullDimSet}
		& \textnormal{dist}(\ego(x), \ob) > \mathsf{d}_\textnormal{min} \\ 
		& \quad \Leftrightarrow~  \exists \lambda\succeq_{\K^*}0, \mu\succeq_{\bar\K^*}0 \colon -g^\top \mu +(At(x) - b)^\top \lambda  > \mathsf{d}_\textnormal{min} ,\ G^\top \mu + R(x)^\top A^\top \lambda = 0,~\|A^\top\lambda\|_* \leq 1. \nonumber
	\end{align}
\else
		\begin{align}\label{eq:distDual_fullDimSet}
		& \textnormal{dist}(\ego(x), \ob) > \mathsf{d}_\textnormal{min} \\ 
		& \quad \Leftrightarrow~  \exists \lambda\succeq_{\K^*}0, \mu\succeq_{\bar\K^*}0 \colon -g^\top \mu +(At(x) - b)^\top \lambda  > \mathsf{d}_\textnormal{min} , \nonumber \\
		& \ \ \ \qquad G^\top \mu + R(x)^\top A^\top \lambda = 0,~\|A^\top\lambda\|_* \leq 1. \nonumber
	\end{align}
\fi
\end{proposition}
\begin{proof}
	Recall that $\dist(\ego(x),\ob) = \min_{e,o}\{ \| e - o \| \colon Ao \preceq_\K b, e\in\ego(x)  \} = \min_{e',o}\{ \| R(x)e' + t(x) - o \| \colon Ao \preceq_\K b, Ge' \preceq_{\bar\K} g  \}$, where the last equality follows from \eqref{eq:transfEgoFullSet}. The dual of this minimization problem is given by $\max_{\lambda,\mu} \{ -g^\top \mu +(At(x) - b)^\top \lambda \colon  G^\top \mu + R(x)^\top A^\top \lambda = 0,~\|A^\top\lambda\|_* \leq 1,~ \lambda\succeq_{\K^*}0,~\mu\succeq_{\bar\K^\star}0  \}$, see e.g., \cite[Section 8.2]{boyd2004convex} for the derivation, where $\|\cdot\|_*$ is the dual norm, and $\K^*$ and $\bar\K^\star$ are the dual cones of $\K$ and $\bar\K$, respectively. Since $\ob$ and $\B$ are assumed to have non-empty relative interior, strong duality holds, and $\dist(\ego(x),\ob) > \mathsf{d}_{\min} ~ \Leftrightarrow ~ \max_{\lambda,\mu} \{ -g^\top \mu +(At(x) - b)^\top \lambda \colon  G^\top \mu + R(x)^\top A^\top \lambda = 0,~\|A^\top\lambda\|_* \leq 1,~ \lambda\succeq_{\K^*}0,~\mu\succeq_{\bar\K^\star}0  \} > \mathsf{d}_{\min}~\Leftrightarrow \exists \lambda\succeq_{\K^*}0,\mu\succeq_{\bar\K^*}0 \colon -g^\top \mu +(At(x) - b)^\top \lambda  > \mathsf{d}_{\min},~G^\top \mu + R(x)^\top A^\top \lambda = 0,~ \|A^\top \lambda\|_* \leq1$.
\end{proof}
Compared to Proposition~\ref{prop:distDual_pointMass}, we see that full-dimensional controlled objects require the introduction of the additional dual variables $\mu^{(m)}$, one for each obstacle $\ob^{(m)}$. By setting $\mathsf{d}_\textnormal{min}=0$, we obtain now the following reformulation of \eqref{eq:MPC_genCollAvoid} for full-dimensional objects:
\begin{equation}\label{eq:MPC_dualDist_fullSet}
	\begin{array}{lll}
		 \displaystyle\min_{\mathbf x, \mathbf u,\bm\lambda, \bm\mu}  &\displaystyle \sum_{k=0}^{N} \ell(x_k, u_k)   \\
			\ \ \text{s.t.} & x_0 = x(0),~ x_{N+1} = x_F, \\
							& x_{k+1} = f(x_k,u_k),~h(x_k, u_k) \leq 0, \\
					&-g^\top\mu_k^{(m)}+(A^{(m)}\,t(x_k) - b^{(m)})^\top \lambda_k^{(m)}  > 0,\\
					& G^\top \mu_k^{(m)} + R(x_k)^\top {A^{(m)}}^\top\lambda_k^{(m)} =0, \\
					& \|{A^{(m)}}^\top\lambda_k^{(m)}\|_* \leq 1, ~\lambda_k^{(m)}\succeq_{\K^*}0  ,~ \mu_k^{(m)}\succeq_{\bar\K^*}0, \\
					& \text{for } k=0,\ldots,N,\ m=1,\ldots,M,
	\end{array}
\end{equation}
where $\lambda_k^{(m)}$ and $\mu_k^{(m)}$ are the dual variables associated with the obstacle $\ob^{(m)}$ at step $k$, $\bm\lambda$ and $\bm\mu$ are the collection of all $\lambda_k^{(m)}$ and $\mu_k^{(m)}$, respectively, and the optimization is performed over $(\mathbf x, \mathbf u, \bm\lambda,\bm\mu)$. Notice that \eqref{eq:MPC_dualDist_fullSet} is an \emph{exact} reformulation of \eqref{eq:MPC_genCollAvoid}. Smoothness of \eqref{eq:MPC_dualDist_fullSet} is ensured if $\|\cdot\|$ is the Euclidean distance, and $\K$ and $\bar{\K}$ are either the standard cone or the second-order cone, see also
Remark~\ref{rem:smoothness} for details.

Similar to the point-mass case in Section~\ref{sec:collFree_pointMass}, the optimal control problem~\eqref{eq:MPC_dualDist_fullSet} is able to generate collision-free trajectories, but  unable to find ``least-intrusive" trajectories in case collision-free trajectories do not exist. This limitation is addressed next.

\subsection{Minimum-Penetration Trajectory Generation}
We overcome the above limitation by considering again the notion of penetration. We begin with the following result:
\begin{proposition}\label{prop:penDual_fullDimSet}
	Assume that the obstacles and controlled object  are given as in \eqref{eq:convOb} and \eqref{eq:transfEgoFullSet}, respectively, and let $\mathsf{p}_\textnormal{max} \geq0 $ be a maximal penetration depth. Then we have:
\ifArXiv
	\begin{align}\label{eq:penDual_fullDimSet}
	& \textnormal{pen}(\ego(x), \ob) < \mathsf{p}_\textnormal{max}  \\ 
	& \quad \Longleftrightarrow~  \exists \lambda\succeq_{\K^*}0,\mu\succeq_{\bar\K^*}0 \colon g^\top \mu +(b-A\,t(x))^\top \lambda  < \mathsf{p}_\textnormal{max} ,\  G^\top \mu + R(x)^\top A^\top \lambda = 0,~\|A^\top\lambda\|_* = 1. \nonumber
	\end{align}

\else
		\begin{align}\label{eq:penDual_fullDimSet}
	& \textnormal{pen}(\ego(x), \ob) < \mathsf{p}_\textnormal{max}  \\ 
	& \quad \Longleftrightarrow~  \exists \lambda\succeq_{\K^*}0,\mu\succeq_{\bar\K^*}0 \colon g^\top \mu +(b-A\,t(x))^\top \lambda  < \mathsf{p}_\textnormal{max} ,\nonumber \\
	& \ \ \quad \qquad G^\top \mu + R(x)^\top A^\top \lambda = 0,~\|A^\top\lambda\|_* = 1. \nonumber
	\end{align}
\fi	
\end{proposition}
\begin{proof}
	It follows from \cite{CameronCulleyTranslationalDistance1986} that $\pen(\ego(x),\ob) = \pen(0, \ob - \ego(x))$, where $ \ob - \ego(x) := \{o-e \colon o\in\ob,~ e\in\ego(x)\}$ is the Minkowski difference. Furthermore, we have from the proof of Proposition~\ref{prop:penDual_pointMass} that $\pen(0,\ob-\ego(x)) = \inf_{\{z\colon \|z\|_*=1\}}\{ \max_{y\in\ob-\ego(x)}\{y^\top z\}  \}$. Using strong duality of convex optimization, we can dualize the inner maximization problem as $\max_{o\in\ob,e\in\ego(x)}\{z^\top(o-e)\} = \max_{o\in\ob,e'\in\B}\{z^\top(o-R(x)e'-t(x))  \} = \min_{\lambda,\mu}\{b^\top\lambda + g^\top\mu - z^\top t(x) \colon A^\top \lambda = z,~G^\top \mu = -R^\top z \colon \lambda\succeq_{K^*}0,~\mu\succeq_{\bar\K^*} 0 \}$. Hence,  $\pen(0, \ob - \ego(x)) = \inf_{z,\lambda,\mu} \{  b^\top\lambda + g^\top \mu - z^\top t(x) \colon A^\top\lambda=z,~G^\top\mu=-R^\top z ,~\|z\|_*=1   \}$. Eliminating the $z$-variable using the first equality constraint and following the steps of the proof of Proposition~\ref{prop:penDual_pointMass} gives the desired result.
\end{proof}
The following theorem shows that Propositions~\ref{prop:distDual_fullDimSet} and \ref{prop:penDual_fullDimSet} can be combined to represent the \emph{signed distance} function.
\begin{theorem}\label{thm:signDist_fullDimSet}
	Assume that the obstacles and controlled object are given as in \eqref{eq:convOb} and \eqref{eq:transfEgoFullSet}, respectively. Then, for any $\mathsf d\in\R$, we have:
\ifArXiv
	\begin{align}\label{eq:signDist_fullDimSet}
	& \textnormal{sd}(\ego(x), \ob) > \mathsf{d} \\ 
	& \quad \Longleftrightarrow~  \exists \lambda\succeq_{\K^*}0,\mu\succeq_{\bar\K^*}0 \colon -g^\top \mu +(At(x) - b)^\top \lambda  > \mathsf{d} ,\ G^\top \mu + R(x)^\top A^\top \lambda = 0,~\|A^\top\lambda\|_* = 1 . \nonumber
	\end{align}
\else
		\begin{align}\label{eq:signDist_fullDimSet}
		& \textnormal{sd}(\ego(x), \ob) > \mathsf{d} \\ 
		& \quad \Longleftrightarrow~  \exists \lambda\succeq_{\K^*}0,\mu\succeq_{\bar\K^*}0 \colon -g^\top \mu +(At(x) - b)^\top \lambda  > \mathsf{d} ,\nonumber \\
		& \ \ \qquad \qquad G^\top \mu + R(x)^\top A^\top \lambda = 0,~\|A^\top\lambda\|_* = 1 . \nonumber
	\end{align}
\fi
\end{theorem}
\begin{proof}
	By definition, $\sd(\ego(x),\ob) = \dist(\ego(x),\ob)$ if $\ego(x) \cap \ob = \emptyset$, and $\sd(\ego(x),\ob) = -\pen(\ego(x),\ob)$ if $\ego(x) \cap \ob \neq \emptyset$. Consider now $\ego(x) \cap \ob \neq \emptyset$, in which case \eqref{eq:signDist_fullDimSet} follows directly from \eqref{eq:penDual_fullDimSet}. Assume now that $\ego(x) \cap \ob = \emptyset$; then we have from \eqref{eq:distDual_fullDimSet} that $\sd(\ego(x),\ob) > \mathsf{d}$ is equivalent to $\exists \lambda\succeq_{\K^*}0,\mu\succeq_{\bar\K^*}0 \colon -g^\top \mu +(At(x) - b)^\top \lambda  > \mathsf{d} ,~G^\top \mu + R(x)^\top A^\top \lambda = 0,~\|A^\top\lambda\|_* \leq 1 $. Due to homogeneity with respect to $\lambda$ and $\mu$, if the previous condition is satisfied, then there also exists a $\lambda'\succeq_{\K^*}0$ and $\mu'\succeq_{\bar K^*}0$ such that $-g^\top \mu +(At(x) - b)^\top \lambda  > \mathsf{d} ,~G^\top \mu + R(x)^\top A^\top \lambda = 0,~\|A^\top\lambda\|_* = 1$. This concludes the proof.
\end{proof}
Theorem~\ref{thm:signDist_fullDimSet} allows us to formulate the following soft-constrained \emph{minimum-penetration} optimal control problem
\begin{equation}\label{eq:MPC_genCollAvoid_fullSet}
	\begin{array}{lll}
		 \displaystyle\min_{\mathbf x, \mathbf u,\mathbf s, \bm\lambda, \bm\mu}  &\displaystyle \sum_{k=0}^{N} \left[ \ell(x_k, u_k) + \kappa\cdot\sum_{m=1}^M s_{k}^{(m)}  \right]  \\
		\ \ \ \text{ s.t.} & x_0 = x_S,~ x_{N+1} = x_F, \\
					& x_{k+1} = f(x_k,u_k),~h(x_k, u_k) \leq 0, \\
					& -g^\top\mu_k^{(m)}+(A^{(m)}\,t(x_k) - b^{(m)})^\top \lambda_k^{(m)}  > -s_{k}^{(m)},\\
					& G^\top \mu_k^{(m)} + R(x_k)^\top {A^{(m)}}^\top\lambda_k^{(m)} =0, \\
					& \|{A^{(m)}}^\top\lambda_k^{(m)}\|_* = 1, \\
					& s_k^{(m)}\geq0,~\lambda_k^{(m)}\succeq_{\K^*}0  ,~\mu_k^{(m)}\succeq_{\bar\K^*}0, \\
					& \text{for } k=0,\ldots,N,\ m=1,\ldots,M.
	\end{array}
\end{equation}
where $s_k^{(m)} \in \R_+$ is the slack variable associated to obstacle $\ob^{(m)}$ at time step $k$, and $\kappa\geq0$ is a weight factor that keeps the slack variable as small as possible. Smoothness of \eqref{eq:MPC_genCollAvoid_fullSet} is ensured if $\|\cdot\|$ is the Euclidean distance, and $\K$ and $\bar{\K}$ are either the standard cone or the second-order cone, see  Remark~\ref{rem:smoothness} for details.

In the following sections, we illustrate our obstacle avoidance formulation on two applications: a quadcopter path planning problem  where the point-mass formulation is used (Section~\ref{sec:ex_quadcopter}), and an automated parking problem, where the full-dimensional obstacle avoidance problem formulation is used (Section~\ref{sec:ex_autParking}).

\section{Example 1: Quadcopter Path Planning}\label{sec:ex_quadcopter}
In this section, we illustrate reformulations \eqref{eq:MPC_dualDist_pointMass} and \eqref{eq:MPC_genCollAvoid_pointMass} on a quadcopter navigation problem, where the quadcopter must find a  path from one end of the room to the other end, while avoiding a low-hanging wall and passing through a small window hole, see Fig.~\ref{fig:TrajQuad}.

\begin{figure}[htb]
        \centering
        \ifArXiv
		\includegraphics[trim = 0mm 00mm 0mm 00mm, clip, width=0.55\textwidth]{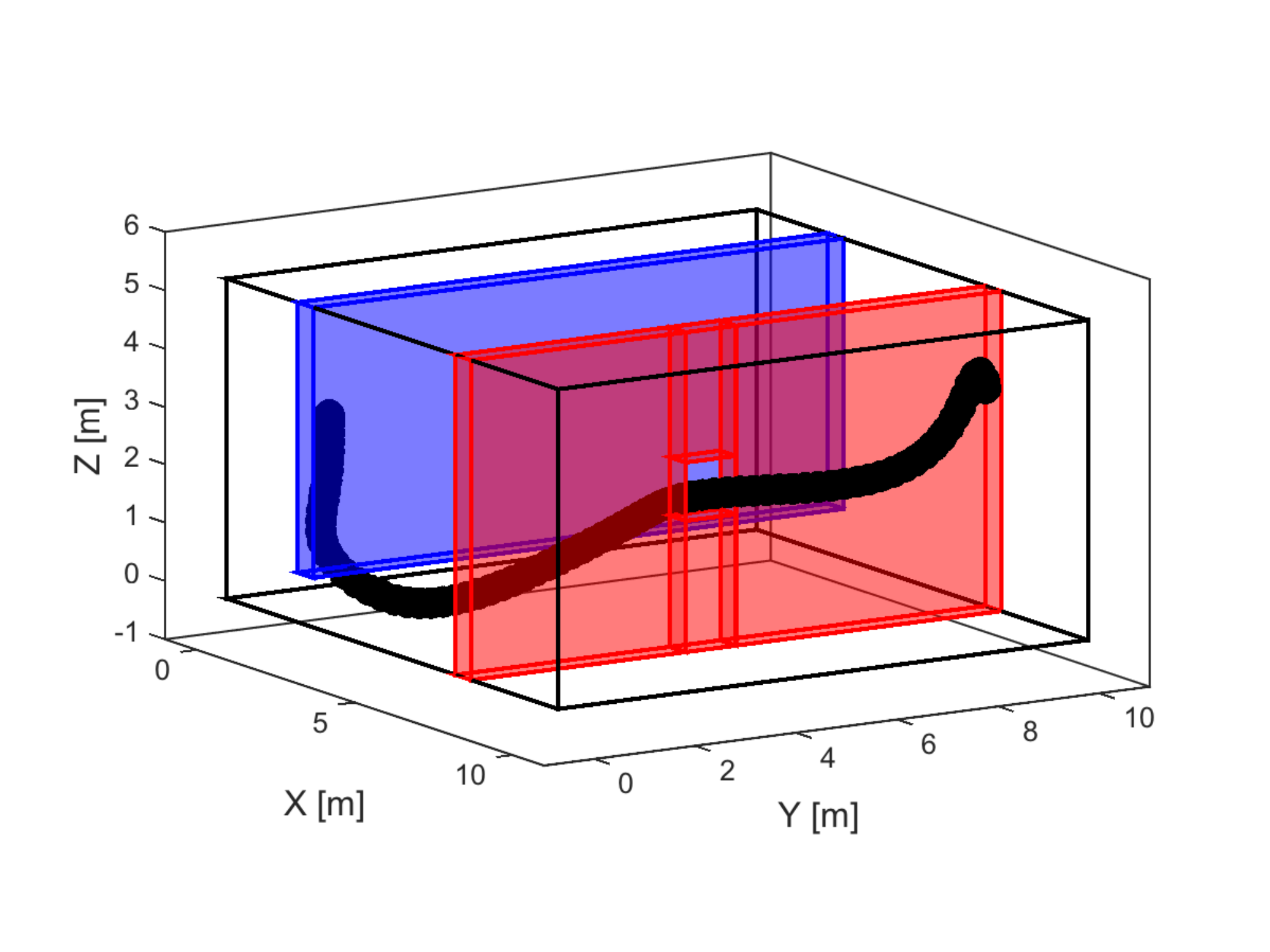}
	\else
		\includegraphics[trim = 0mm 00mm 0mm 00mm, clip, width=0.45\textwidth]{Figures/TrajQuad_3D_1}
	\fi
        \vspace{-0.2cm}
        \caption{Setup of quadcopter example together with a (locally optimal) point-to-point trajectory. The start position is behind the blue wall at $X=1$, and the end position in front of red wall at $X=9$. The quadcopter needs to fly below the blue wall and pass through a window in the red wall, see \url{https://youtu.be/7WLNJhaHcoQ} for an animation. The black circles illustrate the safety distance $\mathsf d_\textnormal{min}=0.25$\,m which models the shape of the quadcopter.}
        \label{fig:TrajQuad}
\end{figure}

\subsection{Environment and Obstacle Modeling}
The size of the room is $10.5 \times 10.5 \times 5.5$\,m, and we
see from Fig.~\ref{fig:TrajQuad} that the direct path between the start and end position is blocked by two obstacles. The first obstacle, a low-hanging wall, blocks the entire upper part of the room, and can only be passed from below. The second obstacle, another wall, blocks the entire room, but has a small window through which the quadcopter must pass to reach its target position. We approximate the shape of the quadcopter by a (Euclidean) sphere of radius 0.25\,m. In the framework of \eqref{eq:MPC_dualDist_pointMass} and \eqref{eq:MPC_genCollAvoid_pointMass}, the shape of the quadcopter can be taken into account by requiring a safety distance of $\mathsf d_\textnormal{min}=0.25$\,m.

The first wall, which can only be passed from below, is placed at $X$ = 2\,m, and the passage below is 0.85\,m high. The second wall is placed at $X$ = 7\,m, and the window (size $1\times1$m) is placed in the middle of the second wall at a height of $Z = 2.5$\,m. Finally the depth of both walls is 0.5\,m. This obstacle formation can be formally formulated using five axis-aligned rectangles, where the first obstacle is represented by one such rectangle and the window can be modeled as the union of four rectangles, see Fig.~\ref{fig:TrajQuad}. The collision avoidance constraints with respect to the four outer walls of the room are achieved by appropriately upper- and lower-bounding the $(X,Y,Z)-$coordinates of the quadcopter.

\subsection{Quadcopter Model}
We consider the standard quadcopter model as used in \cite{Mellinger2012}, which is derived by finding the equation of motion of the center of gravity (CoG) and summarized next. In this model, $X,Y,Z$ denote the position of the CoG in the world frame, and we use the $Z$-$X$-$Y$ Euler angles do describe the rotation of the quadcopter, where $\phi$ is the pitch angle, $\theta$ is the roll angle, and $\psi$ is the yaw angle. The rotation matrix that translates from the world to the body frame, which is defined with respect to the CoG, is hence given by
\begin{align*}
    R_{WB} = \begin{bmatrix}
    c_\psi c_\theta - s_\phi s_\psi s_\theta & -c_\phi s_\psi & c_\psi s_\theta + c_\phi s_\psi s_\theta\\
    s_\psi c_\theta + s_\phi c_\psi s_\theta &  c_\phi c_\psi & s_\psi s_\theta - s_\phi c_\psi c_\theta\\
    - c_\phi s_\theta & s_\phi  & c_\phi c_\theta\\
    \end{bmatrix}\,,
\end{align*}
where $s_\phi := \sin(\phi)$ and $c_\phi := \cos(\phi)$. The accelerations of the CoG can be derived by considering the sum of the forces produced by the four rotors $F_i$ which point in positive z-direction in the body frame, and the gravity force which acts on the negative z-direction in the world frame, resulting in the following equation of motion,
\begin{subequations}
\begin{align}\label{eq:quadEq1}
    m \begin{bmatrix} \ddot{X} \\ \ddot{Y} \\ \ddot{Z} \end{bmatrix} = \begin{bmatrix} 0 \\ 0 \\ -mg \end{bmatrix} + R_{WB} \begin{bmatrix} 0 \\ 0 \\ \sum_{i=1}^4 F_i \end{bmatrix}\,,
\end{align}
where $m$ is the mass of the quadcopter and $\ddot X$, $\ddot Y$ and $\ddot Z$ are the second time derivatives of $X$, $Y$ and $Z$, respectively. The attitude dynamics of the quadcopter is derived in the body rates $p$, $q$, and $r$ which are related to the Euler angles through the following rotation matrix,
\begin{align}\label{eq:quadEq2}
   \begin{bmatrix} \dot \phi \\ \dot \theta \\ \dot \psi \end{bmatrix} = \begin{bmatrix} c_\theta & 0 & s_\theta \\ s_\theta t_\phi & 1 & - c_\theta t_\phi\ \\ -s_\theta/c_\phi & 0 & c_\theta/c_\phi  \end{bmatrix} \begin{bmatrix} p \\ q \\ r \end{bmatrix}\,.
\end{align}
The body rates are given by the following equation of motion, which is driven by the four rotor forces $F_i$, as well as the corresponding moments $M_i$ and has the following form,
\begin{align}\label{eq:quadEq3}
    I \begin{bmatrix} \dot p \\ \dot q \\ \dot r \end{bmatrix} = \begin{bmatrix} L (F_2 - F_4) \\ L (F_3 - F_1) \\ M_1-M_2+M_3-M_4 \end{bmatrix} -  \begin{bmatrix} p \\ q \\ r \end{bmatrix} \times I \begin{bmatrix} p \\ q \\ r \end{bmatrix} \,,
\end{align}
where $I$ is the inertia matrix which in our case is diagonal and $L$ is the distance from the CoG to the rotor. The rotor forces and moments depend quadratically on the motor speed $\omega_i$, which are the control inputs and are defined as follows,
\begin{align}\label{eq:quadEq4}
    F_i = k_F \omega_i ^2\,, \qquad M_i = k_M \omega_i ^2\,,
\end{align}
where, $k_F$ and $k_M$ are constants depending on the rotor blades. Hence, the state of the quadcopter is $x = [X,Y,Z,\phi,\theta,\psi,\dot X,\dot Y,\dot Z,p,q,r]$ 
and the inputs are the four rotor speeds $u = [\omega_1,\omega_2,\omega_3,\omega_4]$.
\end{subequations}

The parameters of the model, as well as the bounds on the inputs, are taken from \cite{MellingerPhD}, which corresponds to a quadcopter which weighs 0.5\,kg and has a diameter of half a meter. Bounds on the angles, velocities and body rates are considered, and the dynamics can be brought into the form \eqref{eq:stateDynamics} using a (forward) Euler discretization, such that $x_{k+1} = x_k + T_\text{opt} \tilde f(x_k,u_k)$, where $T_\text{opt}$ is the sampling time, and $\tilde f(\cdot,\cdot)$ is the continuous-time dynamics that can be obtained from \eqref{eq:quadEq1}--\eqref{eq:quadEq4}.

\subsection{Cost function}
Our control objective is to navigate the quadcopter as fast as possible, while avoiding excessive control inputs. We combine these competing goals as a weighted sum of the form $J = q \tau_F + \sum_{k=0}^{N-1} u_k^T R u_k$,
where $\tau_F$ is the final time and $R = R^\top\succeq 0$, and $q\geq0$ are weighting factors. 
Motivated by \cite{Fahroo2000}, we do not directly minimize $\tau_F$; instead, observing that  $\tau_F = N T_\text{opt}$, we will treat the discretization time $T_\text{opt}$ as a decision variable. This allows the use of the slightly modified cost function 
\begin{equation}\label{eq:cost_used}
    J(\mathbf u, T_\text{opt}) = q N T_{\text{opt}} + \sum_{k=0}^{N-1} u_k^T R u_k
\end{equation} which will be used in the numerical simulations later on.

Treating $T_\text{opt}$ as an optimization variable has the additional benefit that the duration of the maneuver does not need to be fixed a priori, allowing us to avoid feasibility issues caused by a too short maneuver lengths. We point out that having $T_\text{opt}$ as a decision variable  comes at the cost of introducing an additional decision variable $T_\text{opt}$, which renders the dynamics ``more non-linear", which can be seen when looking at the Euler discretization in the previous section.

\subsection{Choice of Initial Guess}
Recall that \eqref{eq:MPC_dualDist_pointMass} and \eqref{eq:MPC_genCollAvoid_pointMass} are non-convex optimization problems, and hence computationally challenging to solve in general. In practice, one has to content oneself with a locally optimal solution that, for instance, satisfied the Karush-Kuhn-Tucker (KKT) conditions \cite{WachterIPOPT2006}, since most numerical solvers operate locally. Furthermore, it is well-known that the solution quality critically depends on the initial guess (``warm starting point") that is provided to the solvers, and that different initial guesses can lead to different (local) optima. Unfortunately, computing a good initial guess is often difficult and highly problem dependent; ideally, the initial guess should be obstacle-free and approximately satisfy the system dynamics.

For the quadcopter example, we have observed that the well-known A$^\star$ algorithm is able to provide good initial guesses. A$^\star$ is a graph search algorithm that is able to find obstacle-free paths by gridding the position space. It is similar to Dijkstra's algorithm, but uses a so-called heuristic function to perform a ``best-first" search, see \cite{Hart1968, Hart1972} for details. In our quadcopter example, we use the A$^\star$ algorithm to find an obstacle-free path in the position space, which we use to initialize the states that correspond to the quadcopter's position. The remaining states are initialized with zero, while inputs are initialized with the steady state input that keeps the quadcopter in a hoovering position. The dual variables $\lambda_k^{(m)}$  are initialized with 0.05, and the discretization time $T_\text{opt}$ with 0.25. Fig.~\ref{fig:TrajQuadWS} depicts the initial guess used to generate the trajectory shown in Fig.~\ref{fig:TrajQuad}. Notice that, due to gridding, the path in Fig.~\ref{fig:TrajQuadWS} exhibits a zigzag pattern.

\begin{figure}[htb]
        \centering
        \ifArXiv
		\includegraphics[trim = 0mm 00mm 0mm 00mm, clip, width=0.55\textwidth]{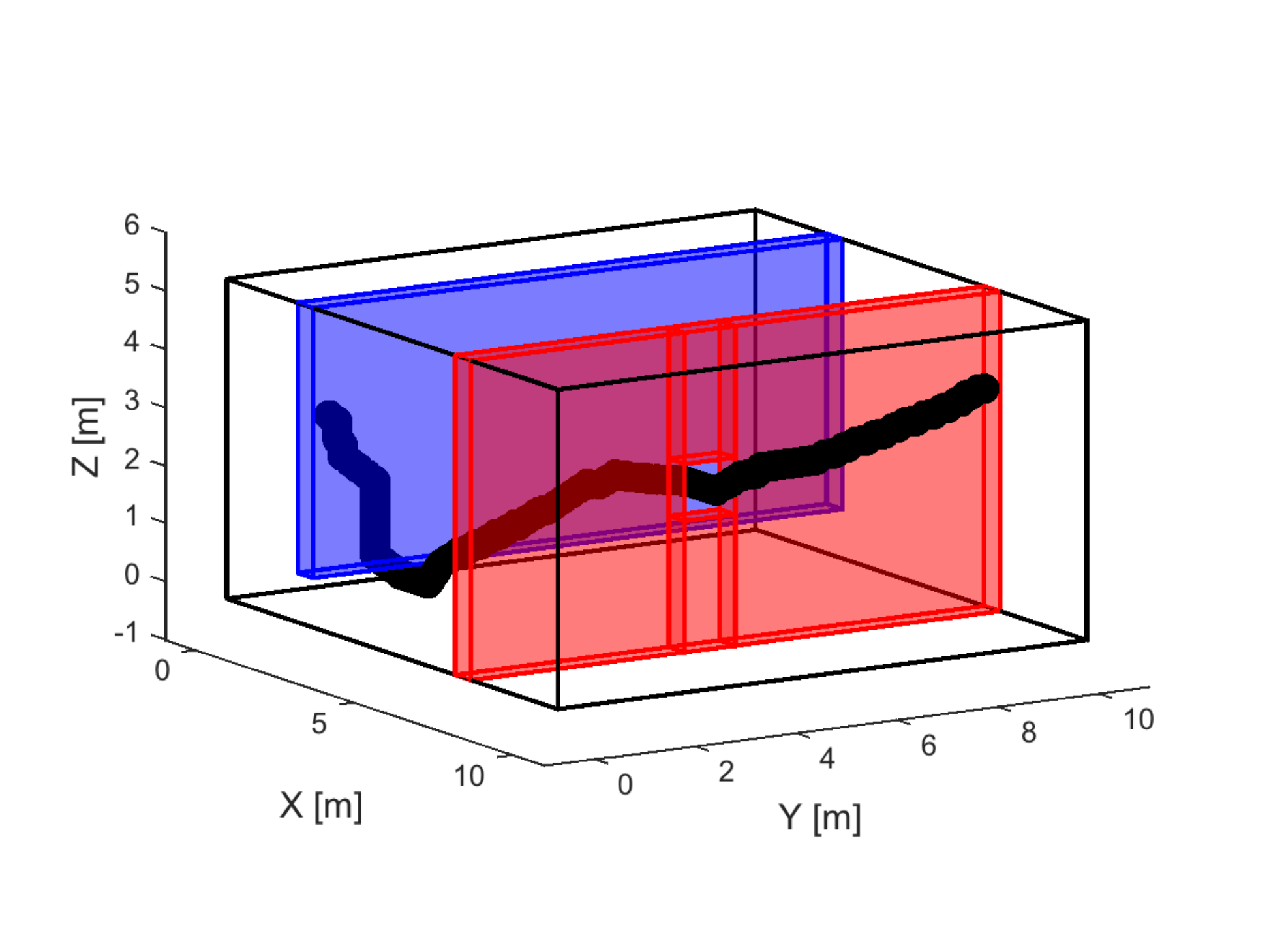}
	\else
		\includegraphics[trim = 0mm 00mm 0mm 00mm, clip, width=0.45\textwidth]{Figures/TrajQuadAs_3D_WS}
	\fi
        \vspace{-0.2cm}
        \caption{A sample warm start trajectory, obtained from the A$^\star$ algorithm, is shown. Notice that this trajectory avoids obstacles but does not satisfy the dynamic constraints of the quadcopter, which will be ``corrected" by solving \eqref{eq:MPC_dualDist_pointMass} and \eqref{eq:MPC_genCollAvoid_pointMass}. The black tube is the safety distance $\mathsf d_\textnormal{min}$ which is used to represent the shape of the quadcopter.}
        \label{fig:TrajQuadWS}
\end{figure}

\subsection{Simulation Results}
To verify the performance and robustness of our approach, we considered 36 path planning scenarios, each starting and ending in a hovering position. The starting point is always located at $(X,Y,Z)=(1,1,3)$\,m, and the finishing point is always located behind the wall with the window at $X=9$\,m, but with varying $Y$ and $Z$ coordinates. The final positions are generated by gridding the $(Y,Z)$ space with nine points in the $Y$ direction and four points in the $Z$ direction as shown in Fig.~\ref{fig:CompTimeQuad}. We tested both the distance formulation \eqref{eq:MPC_dualDist_pointMass} as well as the signed distance formulation \eqref{eq:MPC_genCollAvoid_pointMass}. 
The horizon $N$ equals the number of steps performed by the $A^\star$ algorithm, and takes values between 100 and 129 for the given setup and a grid size of 0.1\,m.
The sampling time $T_\text{opt}$ is restricted to lie between $0.125$\,s and $0.375$\,s. The optimization problems are implemented with the modeling toolbox JuMP in the programming language Julia \cite{DunningHuchetteLubin2017}, and solved using the general purpose nonlinear solver IPOPT \cite{WachterIPOPT2006}. The problems are solved on a 2013 MacBook Pro with an i7 processor clocked at 2.6\,GHz.

Table~\ref{tab:compTimeQuad} lists the minimum, maximum and average computation time of the A$^\star$ algorithm, and the time required to solve problems 
\eqref{eq:MPC_dualDist_pointMass} and \eqref{eq:MPC_genCollAvoid_pointMass}. Fig.~\ref{fig:CompTimeQuad} reports the solution time as a function of the finishing position, where a circle indicates that IPOPT has successfully found a solution. We see from Fig.~\ref{fig:CompTimeQuad} that both the distance and signed distance formulation are able to compute all paths successfully. Interestingly, however, the computation time pattern of these two approaches are not correlated; in other words, a ``difficult" scenario for the distance formulation might be ``easy" for the signed distance formulation, and vice versa. In practice, this implies that to obtain feasible trajectories as fast as possible, the navigation problem should be solved with both obstacle avoidance formulations, and the first solution should be taken. In our setup, such an approach would result in a worst case computation time of 28.9\,s, as opposed to 48.0\,s and 59.1\,s if the distance and signed distance reformulation are considered individually.

We close this section by pointing out that, with a maximum computation time of 2.8\,s, the time for A$^\star$ to find an initial guess is considerably lower than that for solving the optimization problems, see Table~\ref{tab:compTimeQuad}. This is not surprising since A$^\star$ only plans a path in the $(X,Y,Z)$-space and ignores the system dynamics which leads to zigzag behavior, see Fig.~\ref{fig:TrajQuadWS}. A dynamically feasible path is only obtained after solving the optimal control problems \eqref{eq:MPC_dualDist_pointMass} and \eqref{eq:MPC_genCollAvoid_pointMass} which, by explicitly taking into account system dynamics, smoothen and locally optimize the path provided by the A$^\star$ algorithm.

\begin{figure}[htb]
        \centering
        \ifArXiv
	        \includegraphics[trim = 20mm 0mm 20mm 5mm, clip, width=0.6\textwidth]{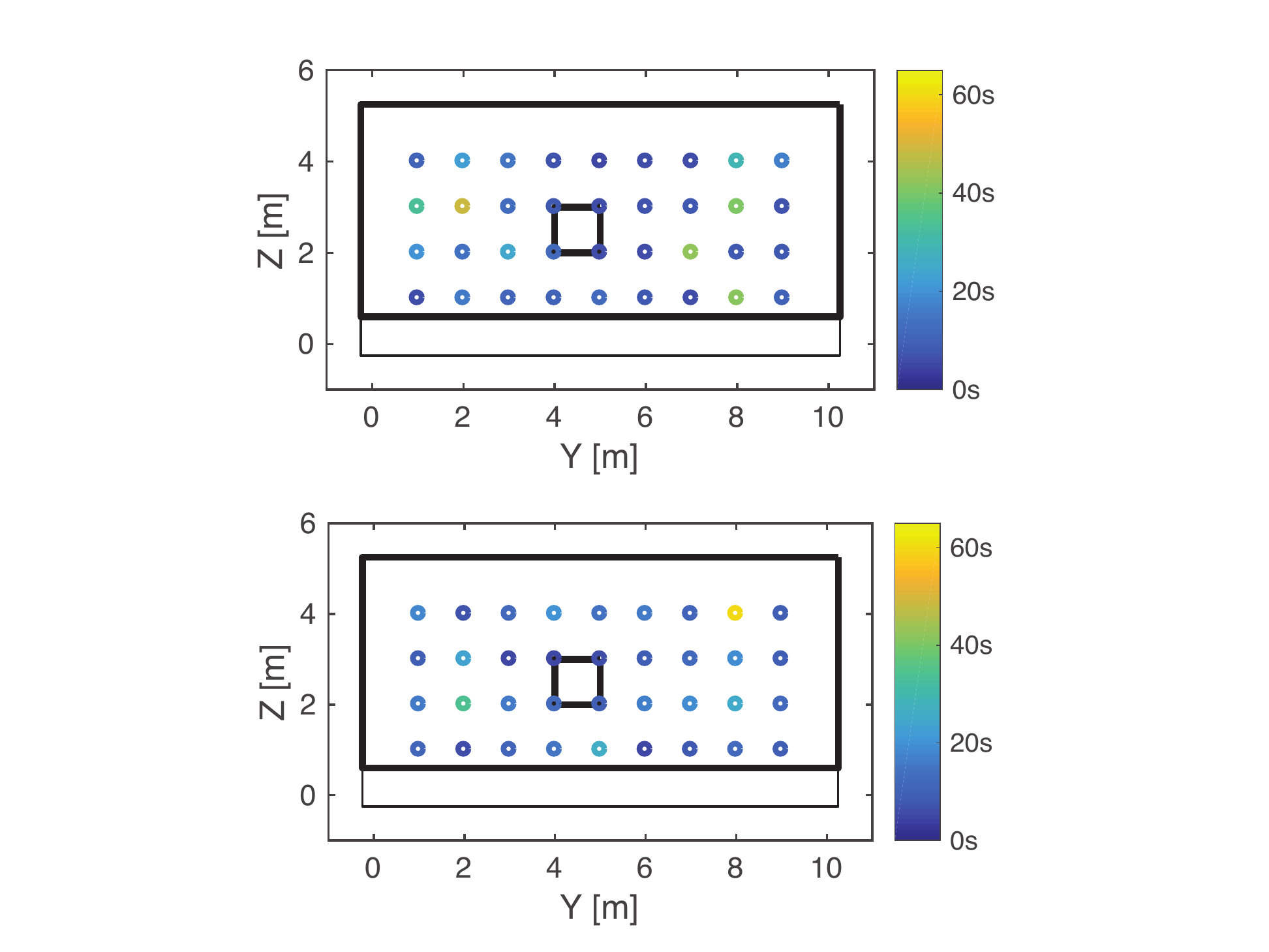}
	\else
		        \includegraphics[trim = 20mm 0mm 20mm 5mm, clip, width=0.5\textwidth]{Figures/CompTimeQuad_HOBCA_vert}
	\fi
        \vspace{-0.5cm}
        \caption{Solution time for quadcopter trajectory planning with distance formulation \eqref{eq:MPC_dualDist_pointMass} (top) and signed-distance formulation \eqref{eq:MPC_genCollAvoid_pointMass} (right).} 
        \label{fig:CompTimeQuad}
\end{figure}

\begin{table}[h]
\caption{Computation time of A$^\star$, distance formulation \eqref{eq:MPC_dualDist_pointMass} and signed distance formulation \eqref{eq:MPC_genCollAvoid_pointMass}.}
\label{tab:compTimeQuad}
\centering 
\begin{tabular}{@{}l c c c @{}}\toprule
\textit{Quadcopter navigation} & min & max & mean\\
 \midrule
warm start (A$^\star$)   & 0.5724\,s & 2.8157\,s &  1.6207\,s  \\ 
distance formulation \eqref{eq:MPC_dualDist_pointMass}               & 4.6806\,s & 47.9762\,s & 14.9716\,s \\ 
signed distance formulation \eqref{eq:MPC_genCollAvoid_pointMass}               &  4.7638\,s & 59.1031\,s & 14.3962\,s \\ 
\bottomrule
\end{tabular}
\end{table}

\section{Example 2: Autonomous Parking}\label{sec:ex_autParking}

As a second application for our collision avoidance formulation, we consider the autonomous parking problem for self-driving cars. In contrast to the quadcopter case, modeling a car as a point-mass and then approximating its shape with a ball can be very conservative and prevent the car from finding a feasible parking trajectory, especially when the environment is tight. In this section, we model the car as a rectangle, and then employ the full-dimensional formulation described in Section~\ref{sec:FullDimObj}. We show that our modelling framework allows us to find obstacle-free parking trajectories even in tight environments. Two scenarios are considered: reverse parking (Fig.~\ref{fig:TrajBack}) and parallel parking (Fig.~\ref{fig:TrajPar2}).

\begin{figure}[htb]
        \centering
        \ifArXiv
		\includegraphics[trim = 0mm 0mm 0mm 0mm, clip, width=0.55\textwidth]{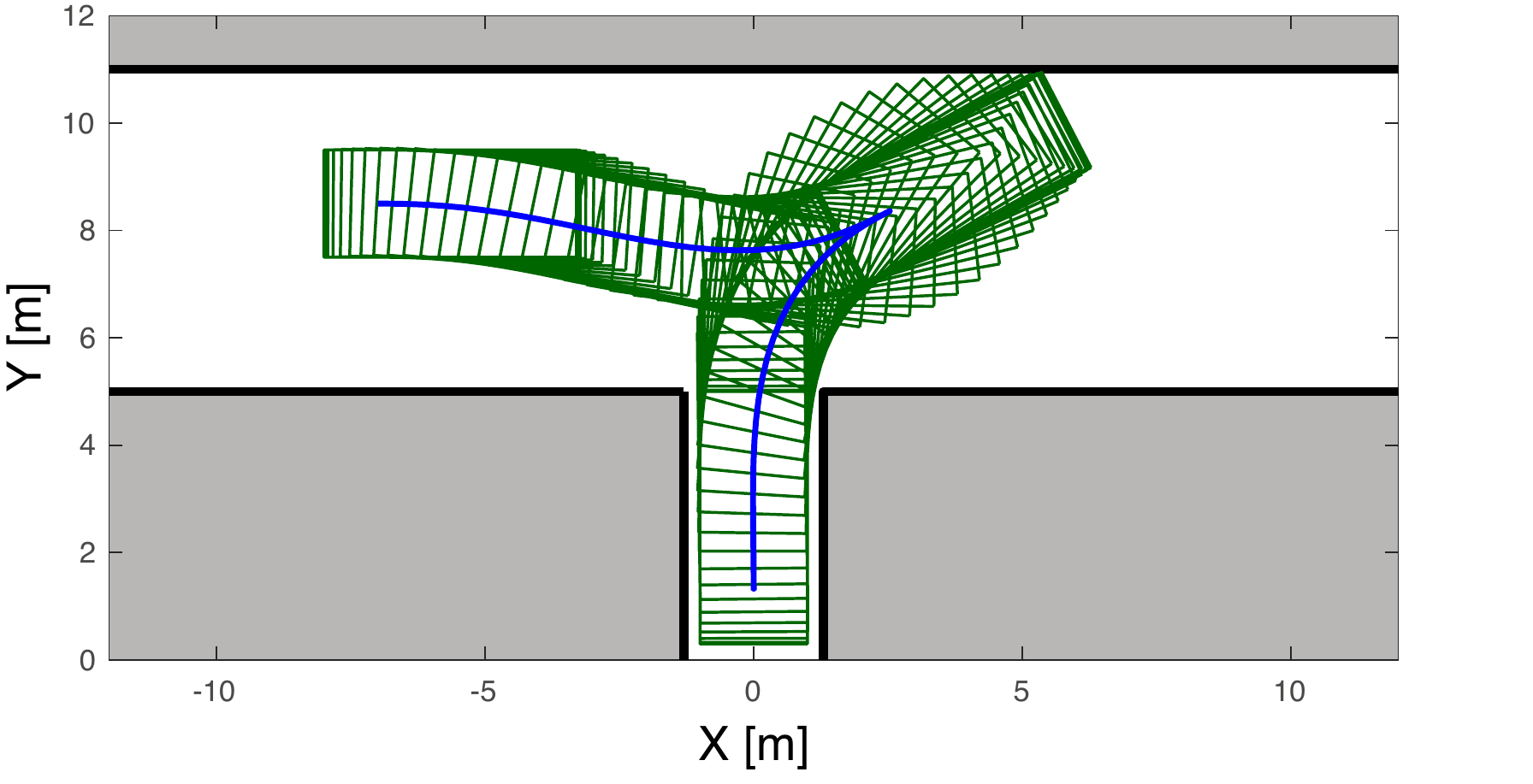}
	\else
		\includegraphics[trim = 0mm 0mm 0mm 0mm, clip, width=0.45\textwidth]{Figures/TrajBack1}
	\fi
        \vspace{-0.2cm}
        \caption{Reverse parking maneuver. The controlled vehicle is shown in green at every time step. Vehicle starts on the left facing to the right, and ends facing upwards, see \url{https://youtu.be/V7IUPW2qDFc} for an animation.}
        \label{fig:TrajBack}
\end{figure}

\begin{figure}[htb]
        \centering
        \ifArXiv
		\includegraphics[trim = 0mm 0mm 0mm 0mm, clip, width=0.55\textwidth]{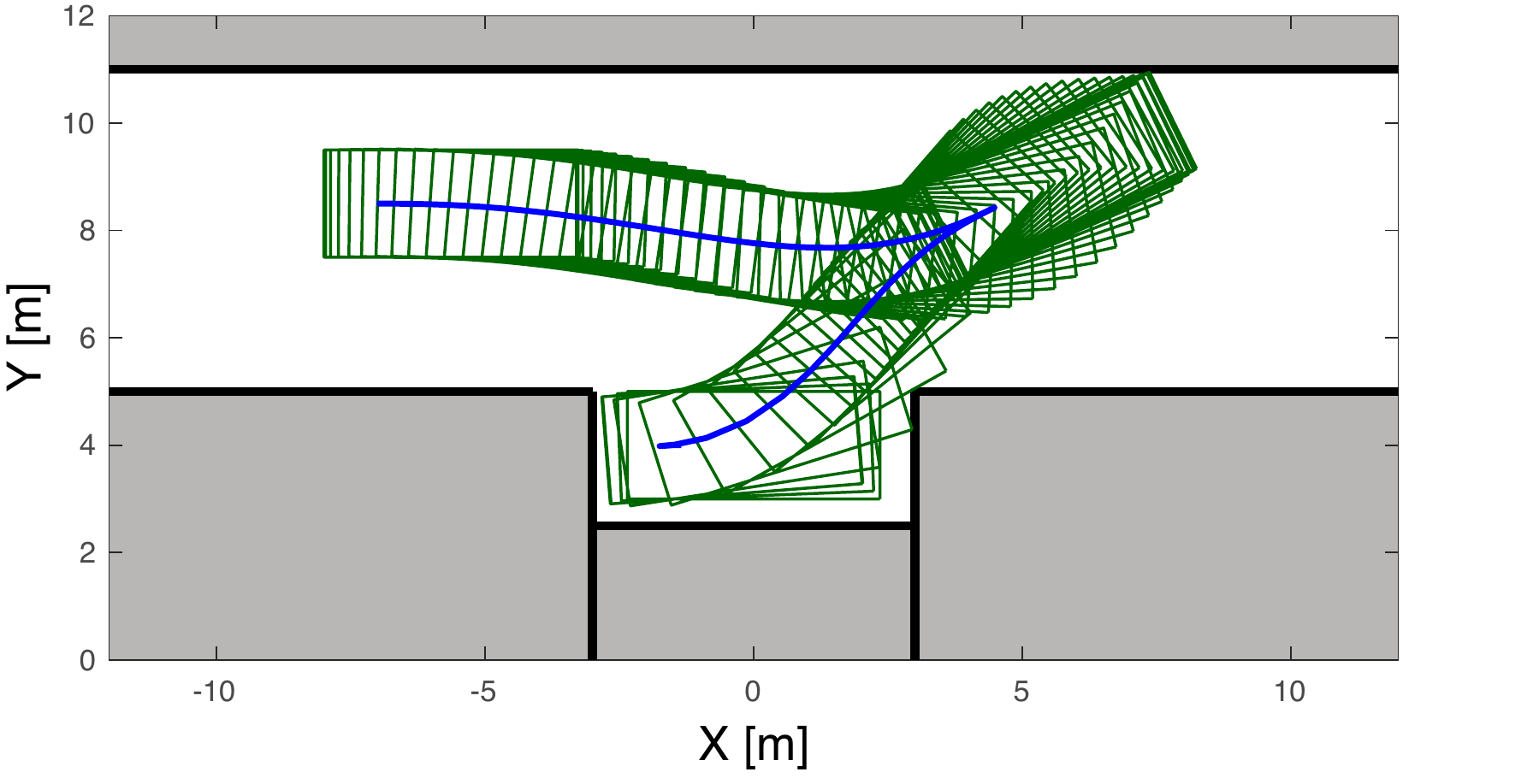}
	\else
		\includegraphics[trim = 0mm 0mm 0mm 0mm, clip, width=0.45\textwidth]{Figures/TrajPar1}
	\fi
        \vspace{-0.2cm}
        \caption{Parallel parking maneuver. The controlled vehicle is shown in green at every time step. Vehicle starts facing to the right, and ends facing to the right, see \url{https://youtu.be/FST7li4M6lU} for an animation.}
        \label{fig:TrajPar2}
\end{figure}

\subsection{Environment and Obstacle Modeling}
For the reverse parking scenario, the parking spot is assumed 2.6\,m wide and 5.2\,m  long. The width of the road, where the car can maneuver in, is 6\,m, see Fig.~\ref{fig:TrajBack} for an illustration. For the parallel parking scenario, the parking spot is 2.5\,m deep and 6\,m long, and the space to maneuver is 6\,m wide (Fig.~\ref{fig:TrajPar2}). Note that the obstacles in the reverse parking scenario can be described by three axis aligned rectangles, while the obstacles in the parallel parking scenario can be described by four axis aligned rectangles. In both cases, the controlled vehicle is modeled as a rectangle of size $4.7 \times 2$\,m, whose orientation is determined by the car's yaw angle.

\subsection{System Dynamics and Cost Function}
The car is described by the classical kinematic bicycle model, which is well-suited for  velocities used in typical parking scenarios. The states $(X,Y)$ correspond to the center of the rear axes, while $\varphi$ is the yaw angle with respect to the X-axis, and $v$ is the velocity with respect to the rear axes. The inputs are the steering angle $\delta$ and the acceleration $a$. Hence, the continuous-time dynamics of the car is given by
\begin{align*}
\dot{X} &= v \cos(\varphi)\,, \\
\dot{Y} &= v \sin(\varphi)\,, \\
\dot{\varphi} & = \frac{v \tan(\delta)}{L}\,, \\
\dot{v} &= a\,,
\end{align*}
where $L=2.7$\,m is the wheel base of the car. The steering angle is limited between $\pm 0.6$\,rad (approximately 34 deg), with rate constraints $\dot\delta\in[-0.6,0.6]$\,rad/s; acceleration is limited to be between $\pm 1$\,m/s$^2$. 
We limit the car's velocity to lie between $-1$ and $2$\,m/s. Similar as in the quadcopter case, the continuous-time dynamics are discretized using a forward Euler scheme. Finally, the same cost function as in the quadcopter is used, i.e., a weighted sum between the discretization-time and control effort is considered.

\subsection{Initial Guess}
Similar to the previous example, the solution quality of the non-convex optimization problems heavily depends on the initial guess provided to the numerical solvers. Unfortunately, it turns out that the A$^\star$ algorithm used in the quadcopter example generally provides a poor warm start as it is unable to take into account the vehicle's non-holonomic dynamics\footnote{Roughly speaking, A$^\star$ will return trajectories that would require the vehicle to move sideways. Simulations indicate that the numerical solvers are typically not able to ``correct" such a behavior and unable to recover a feasible solution when initialized with A$^\star$.}. To address this issue, we resort to a modified version of A$^\star$, called Hybrid A$^\star$ \cite{DolgovThrunPathPlanning2010}. The main idea behind Hybrid A$^\star$ is to use a simplified vehicle model with states $(X,Y,\varphi)$, and a finite number of steering inputs to generate a coarse parking trajectory. Like A$^\star$, Hybrid A$^\star$ grids the state space and performs a tree search, where the nodes are expanded using the simplified vehicle model. We refer the interested reader to \cite{DolgovThrunPathPlanning2010} for details on Hybrid A$^\star$. Fig.~\ref{fig:TrajBackWS} and Fig.~\ref{fig:TrajParWS2} depict two trajectories obtained from the Hybrid A$^\star$ algorithm. Notice that, due to discretization of state and input, the paths generated by Hybrid A$^\star$ seems more ``bang-bang" and less ``smooth" than those shown in Fig.~\ref{fig:TrajBack} and Fig.~\ref{fig:CompTimeQuad}.

\begin{figure}[htb]
        \centering
        \ifArXiv
		\includegraphics[trim = 0mm 0mm 0mm 0mm, clip, width=0.55\textwidth]{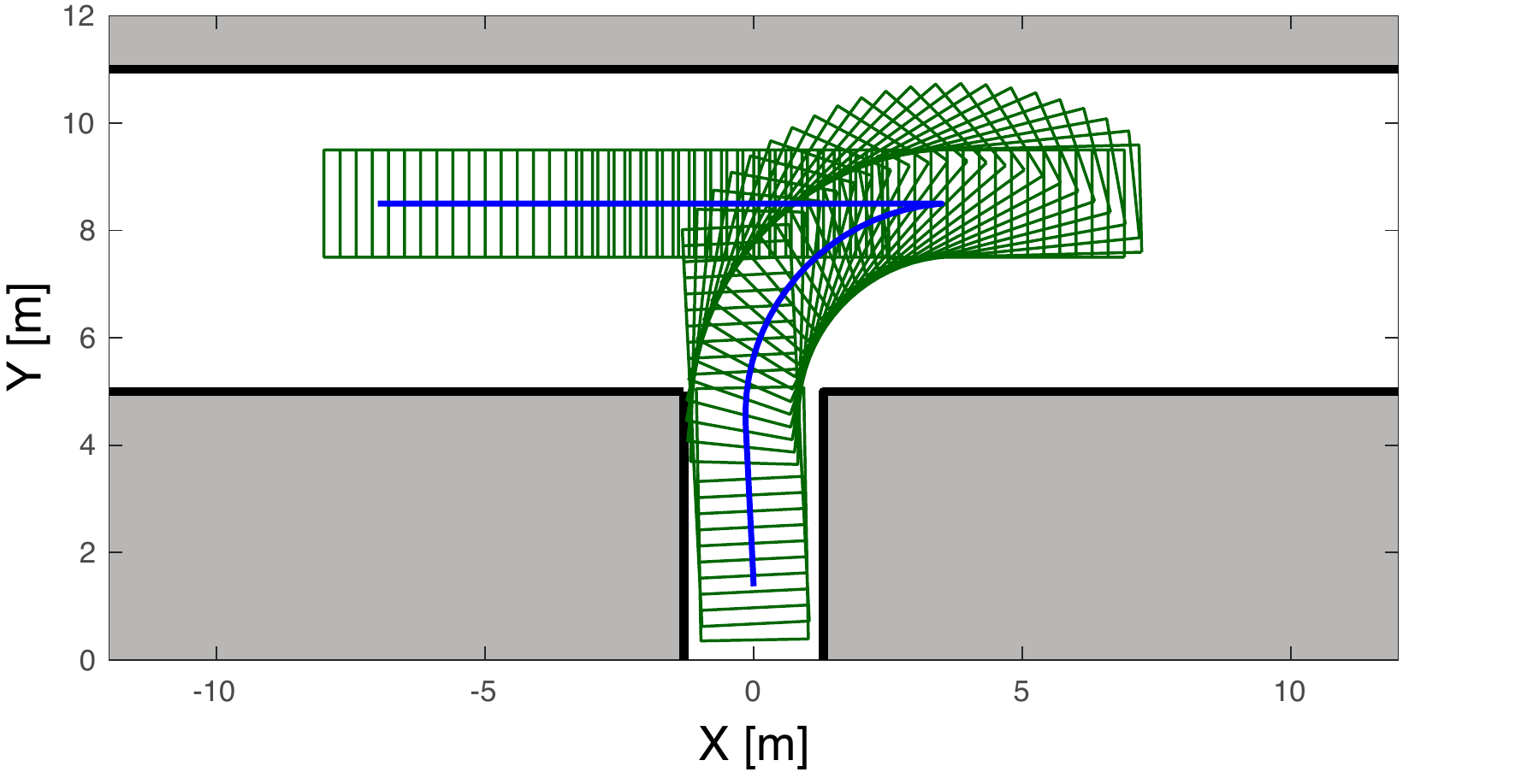}
	\else
		\includegraphics[trim = 0mm 0mm 0mm 0mm, clip, width=0.45\textwidth]{Figures/TrajBack1_WS}
	\fi
        \vspace{-0.2cm}
        \caption{Initial guess provided by Hybrid A$^\star$ for reverse parking.}
        \label{fig:TrajBackWS}
\end{figure}

\begin{figure}[htb]
        \centering
        \ifArXiv
	        \includegraphics[trim = 0mm 0mm 0mm 0mm, clip, width=0.55\textwidth]{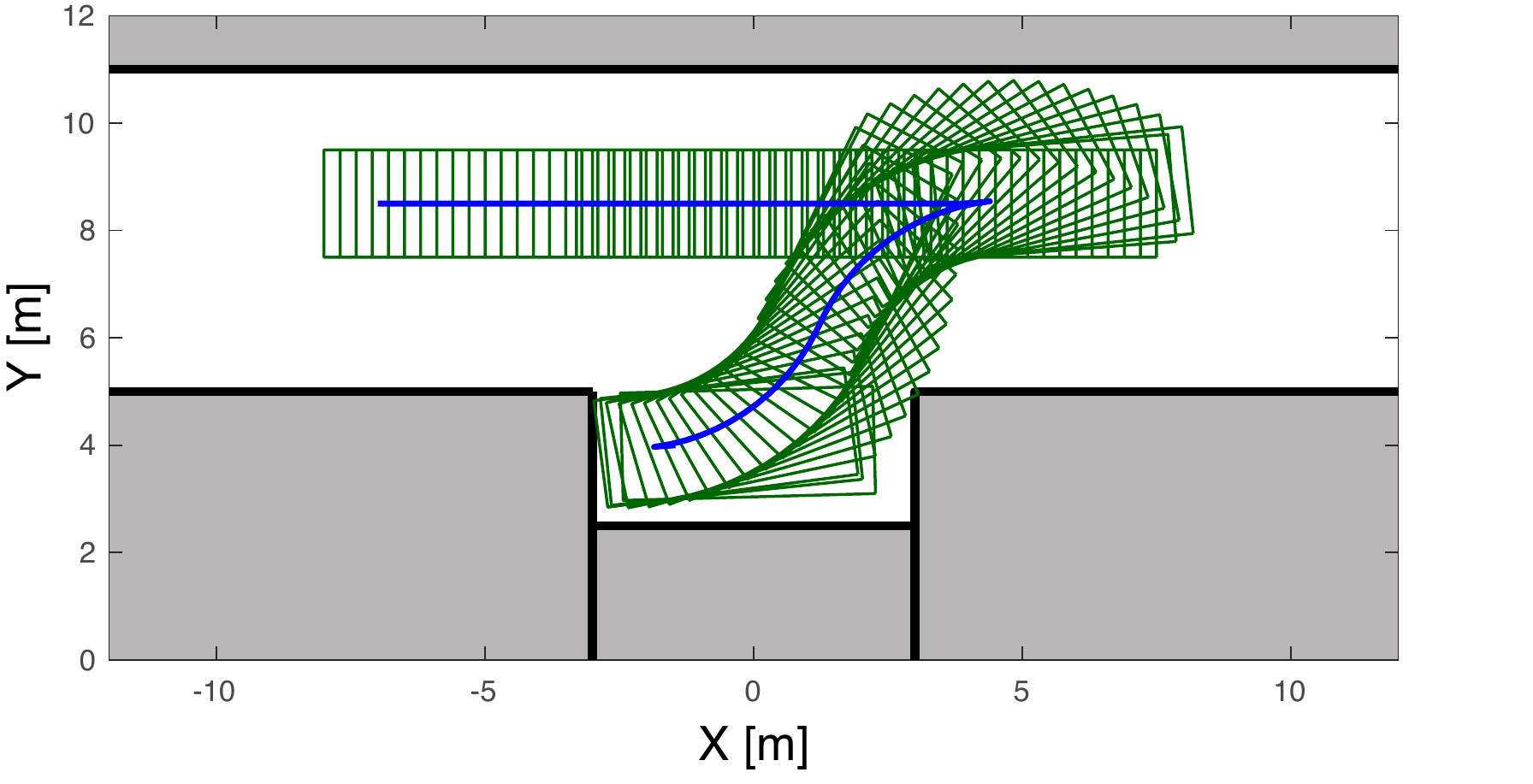}
	\else
	        \includegraphics[trim = 0mm 0mm 0mm 0mm, clip, width=0.45\textwidth]{Figures/TrajPar1_WS}
\	\fi
        \vspace{-0.2cm}
        \caption{Initial guess provided by Hybrid A$^\star$ for parallel parking.}
        \label{fig:TrajParWS2}
\end{figure}

\subsection{Simulation Results}
To evaluate the performance of formulations \eqref{eq:MPC_dualDist_fullSet} and \eqref{eq:MPC_genCollAvoid_fullSet}, we study the reverse and parallel trajectory planning problem.
For both cases, we consider different starting positions but one fixed end position at $X=0$\,m, and investigate the computation time of each method. The starting positions are generated by gridding the maneuvering space within $X\in [-10,10]$\,m and $Y \in [6.5,9.5]$\,m, with 21 grid points in the $X$ direction and 4 grid points in $Y$ direction, see Fig.~\ref{fig:TrajBackCompTime}. The orientation for all the starting points is $\varphi = 0$, resulting in a total of 84 starting points. The horizon length $N$ is given by the Hybrid A$^\star$ algorithm. The optimization problems are again implemented with the modeling toolbox JuMP in the programming Julia \cite{DunningHuchetteLubin2017}, and IPOPT \cite{WachterIPOPT2006} is used as the numerical solver. The problems are solved on a 2013 MacBook Pro with a i7 processor clocked at 2.6 GHz. A Julia-based example code can be found at \url{https://github.com/XiaojingGeorgeZhang/OBCA}.

We begin by considering the reverse parking case, where one specific maneuver is illustrated in Fig.~\ref{fig:TrajBack}. The computation times for the distance and the signed distance formulation are listed in Table~\ref{tab:compTimeParking} (upper half) and shown in Fig.~\ref{fig:TrajBackCompTime}, for all 84 initial conditions. Table~\ref{tab:compTimeParking} indicates that the distance formulation is generally faster than the signed-distance formulation, with a mean computation time of 0.60\,s compared to 1.03\,s. This is not surprising since the signed distance formulation has more decision variables due to the presence of the slack variables $s_k^{(m)}$, see \eqref{eq:MPC_genCollAvoid_fullSet}. Furthermore, we see from Fig.~\ref{fig:TrajBackCompTime} that both approaches are able to find feasible parking trajectories, for all 84  considered initial conditions.
Interestingly, we see that there are no obvious relations between starting positions and solution times.

\begin{figure}[htb]
        \centering
        \ifArXiv
		\includegraphics[trim = 0mm 0mm 0mm 0mm, clip, width=0.50\textwidth]{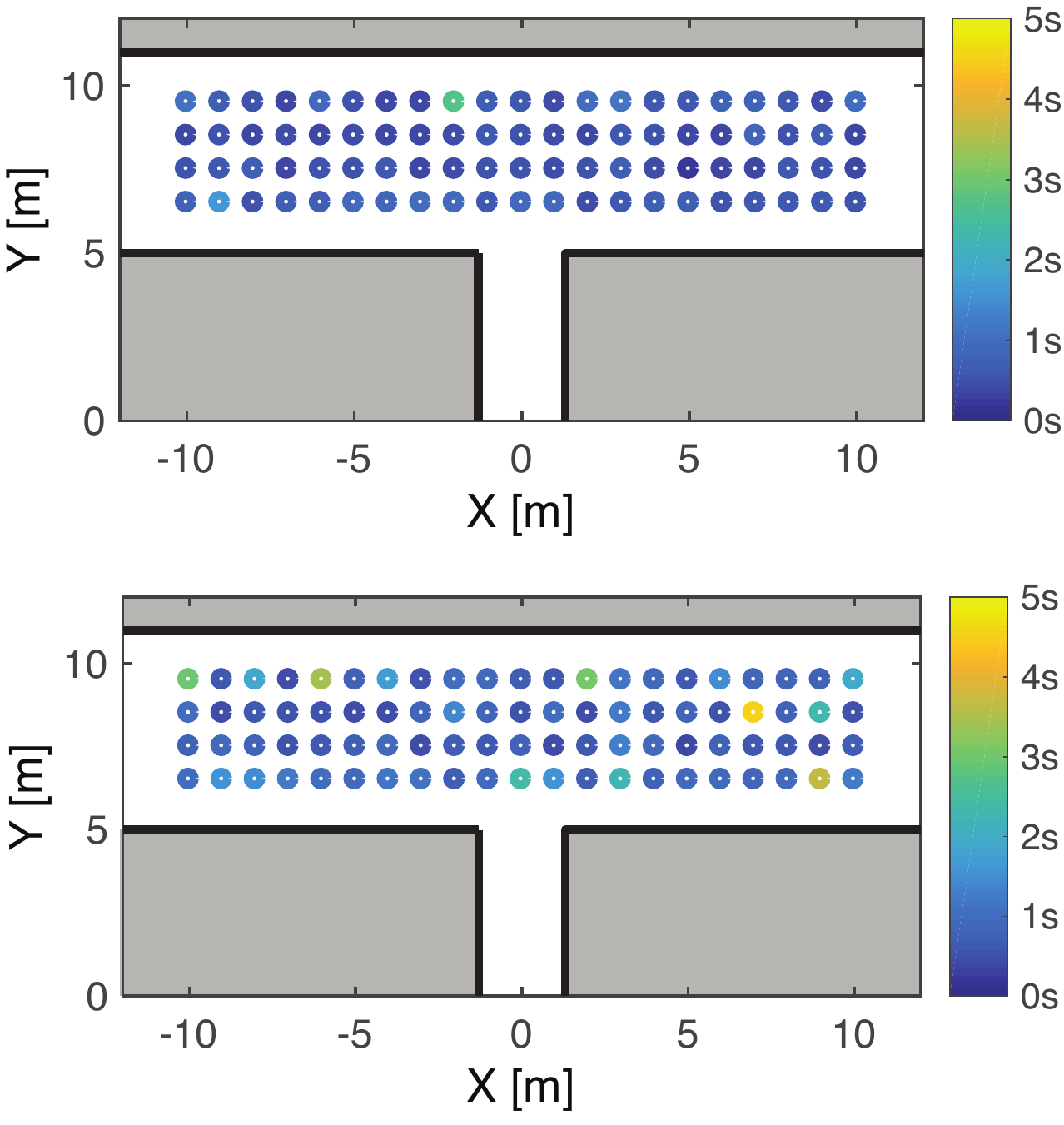}
	\else
		\includegraphics[trim = 0mm 0mm 0mm 0mm, clip, width=0.40\textwidth]{Figures/CompTimeParkingBack_HOBCA}
	\fi
        \vspace{-0.2cm}
        \caption{Solution time for reverse parking with distance formulation \eqref{eq:MPC_dualDist_fullSet} (top) and signed-distance formulation \eqref{eq:MPC_genCollAvoid_fullSet} (bottom).}
        \label{fig:TrajBackCompTime}
\end{figure}

The computation times of the parallel parking case is shown in Fig.~\ref{fig:TrajParCompTime} and Table~\ref{tab:compTimeParking} (lower half). Similar as in the reverse parking case, we see that both approaches have a 100\% success rate, and that, again due to the presence of the slack variables, the signed distance formulation requires longer computation time (1.67\,s on average) than the distance formulation (0.87\,s on average). Compared to reverse parking we see that parallel parking is computationally more demanding. We believe that this is due to the fact that the paths in parallel parking are generally longer than in reverse parking, since the car first needs to drive to the right before it can back into the parking lot, see also Fig.~\ref{fig:TrajPar2}.

\begin{table}[h]
\caption{Computation time of Hybrid A$^\star$, distance formulation \eqref{eq:MPC_dualDist_fullSet} and signed distance formulation \eqref{eq:MPC_genCollAvoid_fullSet}.}
\label{tab:compTimeParking}
\centering 
\begin{tabular}{@{}l c c c @{}}\toprule
& min & max & mean\\
 \midrule
 \textit{Reverse Parking} & & & \\
warm start (Hybrid A$^\star$)   & 0.0315\,s &   3.2230\,s & 0.5491\,s  \\ 
distance formulation \eqref{eq:MPC_dualDist_fullSet}  & 0.2111\,s & 2.7166\,s &  0.6046\,s \\ 
signed distance formulation \eqref{eq:MPC_genCollAvoid_fullSet}  &   0.3200\,s &  4.4840\,s  & 1.0344\,s \\ \midrule
 \textit{Parallel Parking} & & & \\
warm start (Hybrid A$^\star$)    & 0.0421\,s & 2.4766\,s & 0.3012\,s  \\ 
distance formulation \eqref{eq:MPC_dualDist_fullSet}  & 0.2561\,s &  3.9885\,s & 0.8682\,s \\ 
signed distance reformulation \eqref{eq:MPC_genCollAvoid_fullSet} & 0.3850\,s & 6.7266\,s  &1.6703\,s \\ 
\bottomrule
\end{tabular}
\end{table} 

\begin{figure}[htb]
        \centering
        \ifArXiv
		\includegraphics[trim = 0mm 0mm 0mm 0mm, clip, width=0.55\textwidth]{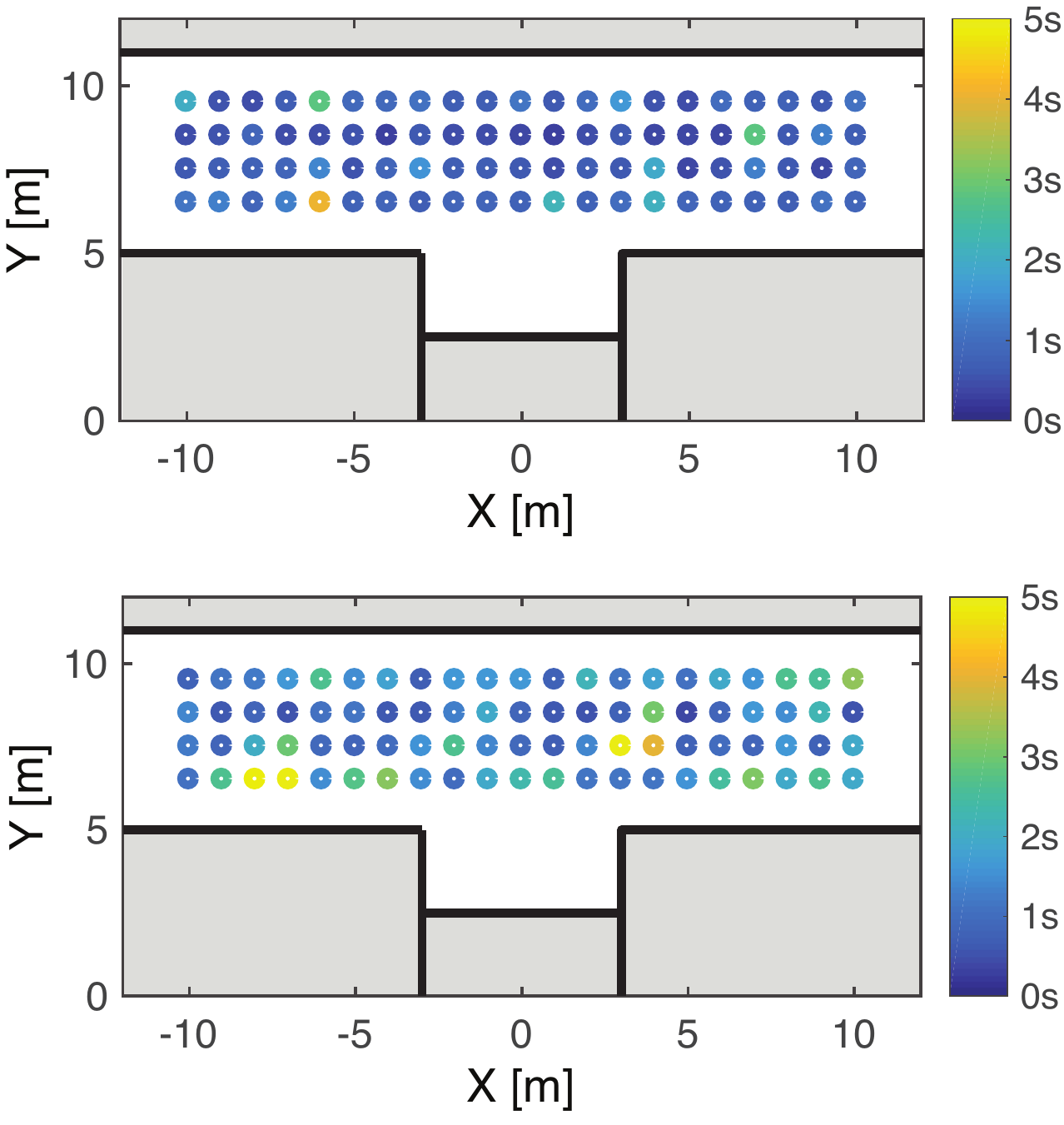}
	\else
		\includegraphics[trim = 0mm 0mm 0mm 0mm, clip, width=0.45\textwidth]{Figures/CompTimeParkingPar_HOBCA}
	\fi
        \vspace{-0.2cm}
        \caption{Solution time for parallel parking with distance formulation \eqref{eq:MPC_dualDist_fullSet} (top) and signed-distance formulation \eqref{eq:MPC_genCollAvoid_fullSet} (bottom).      }
        \label{fig:TrajParCompTime}
\end{figure}

We close this section with the following two remarks: First, we point out that, while the paths generated by the Hybrid A$^\star$ are collision-free and kinodynamically feasible, they are challenging to track with low-level path following controllers because they do not incorporate information on the velocity and do not take into account the rate constraints  in both steering and acceleration, allowing the car to take ``aggressive" maneuveures.
As demonstrated in \cite{ZhangLinigerBorrelli_cdc2018}, this leads, in general, to significantly longer maneuvering times. Second, we notice from Table~\ref{tab:compTimeParking} that the computation time of Hybrid A$^\star$ is comparable to those of (signed) distance. Furthermore, the maximum overall computation time  of Hybrid A$^\star$ and signed distance reformulation is  7.7\,s (reserve parking), and 9.2\,s (parallel parking). This implies that, when initialized with Hybrid A$^\star$, the proposed collision avoidance framework enables real-time autonomous parking in tight environments.

\section{Conclusion}\label{sec:conclusion}
In this paper, we presented smooth reformulations for collision avoidance constraints for problems where the controlled object and the obstacle can be represented as the finite union of convex sets. We have shown that non-differentiable polytopic obstacle constraints can be dealt with via dualization techniques to preserve differentiability, allowing the use of gradient- and Hessian-based optimization methods. The presented reformulation techniques are exact and non-conservative, and apply equally to point-mass and full-dimensional controlled vehicles. Furthermore, in case collision-free trajectories cannot be generated, our framework allows us to find least-intrusive trajectories, measured in terms of penetration.

Our numerical studies, performed on a quadcopter trajectory planning and autonomous car parking example, indicate that, when appropriately initialized, the proposed framework is robust, real-time feasible, and able to generate dynamically feasible trajectories. Furthermore, we have seen that the initialization method is problem-dependent, and should be chosen depending on the system at hand. Current research focuses on appropriately warm starting the discretization time $T_\text{opt}$, as well as on methods for further speeding up computation times.

\section*{Appendix: Proof of Proposition~\ref{prop:penDual_pointMass}}\label{app:proofs}
We need the following standard results from convex analysis:
\begin{lemma}\label{lem:suppHypPlane}
	Let $\mathbb C \subset\R^{n}$ be a compact convex set. 
	\begin{enumerate}
	    \item[$(i)$] Then, $\mathcal H_{\mathbb C} (z) := \{x\in\R^n \colon z^\top x \leq \max_{y\in\mathbb C} y^\top z\}$ is a supporting half space with normal vector $z$, and  
	    \begin{equation}\label{eq:suppHypPlane}
		    \mathbb C = \bigcap_{z\colon \|z\|=1} \mathcal H_{\mathbb C}(z),
	    \end{equation}
	    for any norm $\| \cdot \|$.
	    \item[$(ii)$] Let $\partial \mathcal H_\mathbb{C}(z) := \{x\in\R^n \colon z^\top x = \max_{y\in\mathbb C} y^\top z\}$ be the supporting hyperplane with normal vector $z$. Then, for any $\bar x\in\mathbb{C}$, it holds that
	    \begin{equation}\label{eq:distSuppHypPlane}
	        \dist(\bar x,\partial\mathcal{H}_\mathbb{C}(z)) = \frac{\max_{y\in\mathbb{C}}\{ y^\top z \} - z^\top \bar x} {\|z\|_*},
	    \end{equation}
	    where $\text{dist}(\cdot,\cdot)$ is defined as in \eqref{eq:defDist}.
	\end{enumerate} 
\end{lemma}

\subsubsection*{Proof of Proposition~\ref{prop:penDual_pointMass}}
	First observe that $\pen(\ego(x),\ob) = \dist(\ego(x), \ob^\complement)$, where $\mathbb O^\complement \subset\R^n$ denotes the complement of the set $\mathbb O$. Using this relationship and recalling the definition of $\dist(\cdot,\cdot)$, it follows from \eqref{eq:suppHypPlane} that  $\pen(\ego(x),\ob) = \inf_{\{z\colon \|z\|_*=1\}} \dist(\ego(x), \partial\mathcal H_\mathbb{O}(z))$, where we exploited the fact that \eqref{eq:suppHypPlane} holds for any norm\footnote{Geometrically speaking, $\pen(\ego(x),\ob)$ is the minimum distance between $\ego(x)$ and any supporting hyperplane $\partial\mathcal H_\ob(z)$ of $\mathbb O$.}, and hence also the dual norm $\|\cdot\|_*$. 
	Furthermore, it follows from \eqref{eq:distSuppHypPlane} that, since $\|z\|_*=1$, $\dist(\ego(x),\partial\mathcal H_\mathbb{O}(z)) = \max_{y\in\mathbb O}\{y^\top z\} - z^\top \ego(x)$,  which allows us to rewrite the penetration function as $$\pen(\ego(x),\ob) = \inf_{\{z\colon \|z\|_*=1\}} \{  \max_{y\in\mathbb O}\{y^\top z\} - z^\top \ego(x)   \}.$$
	To see that the min-max problem is equivalent to  \eqref{eq:penDual_pointMass}, we use strong duality of convex optimization to reformulate the inner maximization problem as $\max_{y\in\mathbb O}\{y^\top z\} = \min_{\lambda} \{b^\top \lambda \colon A^\top \lambda = z,~ \lambda \succeq_{\K^*} 0\}$. Hence, $\pen(\ego(x),\ob) = \inf_{z,\lambda} \{b^\top \lambda - z^\top \ego(x) \colon \|z\|_*=1,~A^\top \lambda = z,~\lambda\succeq_{\K^*}0\} = \inf_{\lambda}\{(b-A\ego(x))^\top\lambda \colon \|A^\top \lambda\|_*=1,~\lambda\succeq_{\K^*}0\}$. Finally, we have $\pen(\ego(x),\ob) < \mathsf{p}_{\max} \Leftrightarrow \inf_{\lambda}\{(b-A\ego(x))^\top\lambda \colon \|A^\top \lambda\|_*=1,~\lambda\succeq_{\K^*}0\} < \mathsf{p}_{\max} \Leftrightarrow \lambda\succeq_{\K^*}0 \colon (b-A\,\ego(x) )^\top \lambda  < \mathsf{p}_\textnormal{max}$, which concludes the proof.

\ifArXiv
	\bibliographystyle{unsrt}
	\bibliography{library_GXZ}
\else
        \bibliographystyle{IEEEtran}   
        \bibliography{library_GXZ}
        \balance
        \begin{IEEEbiography}[{\includegraphics[width=1in,height=1.25in,clip,keepaspectratio]{Figures/GZ.jpg}}]{Xiaojing (George) Zhang} received his B.Sc and M.Sc. degrees in electrical engineering and information technology, in 2010 and 2012 from the Swiss Federal Institute of Technology (ETH) Zurich, Switzerland. In January 2017, he received his Ph.D. degree in automatic control, also from ETH Zurich. He is currently a post-doctoral fellow and the Associate Director at the Hyundai Center of Excellence at the University of California, Berkeley. His research interests include modeling, analysis, and control of stochastic uncertain systems, randomized algorithms, machine learning, as well as robust and stochastic optimization. Applications include autonomous vehicles, energy-efficient buildings and electric power systems.
        \end{IEEEbiography}
        
        \begin{IEEEbiography}[{\includegraphics[width=1in,height=1.25in,clip,keepaspectratio]{Figures/AlexLiniger.jpg}}]{Alexander Liniger} (M'14) received the B.Sc. and M.Sc. degrees in mechanical engineering from the Department of Mechanical and Process Engineering, ETH Zurich, Switzerland, in 2010 and 2013, respectively, and is currently pursuing the Ph.D. degree at the Automatic Control Laboratory, ETH Zurich, Switzerland. His research interests include model predictive control, viability theory as well as game theory and their application to autonomous driving and racing.
        \end{IEEEbiography}
        
        \begin{IEEEbiography}[{\includegraphics[width=1in,height=1.25in,clip,keepaspectratio]{Figures/Borrelli_Francesco}}]{Francesco Borrelli} received the Laurea degree in computer science engineering from the University of Naples Federico II, Naples, Italy, in 1998, and the Ph.D. degree from ETH Zurich, Zurich, Switzerland, in 2002.
        
        He is currently a Professor with the Department of Mechanical Engineering, University of California, Berkeley, CA, USA. He is the author of more than 100 publications in the field of predictive control and author of the book \textit{Constrained Optimal Control of Linear and Hybrid Systems} (Springer Verlag). His research interests include constrained optimal control, model predictive control and its application to advanced automotive control and energy efficient building operation.
        
        Dr. Borrelli was the recipient of the 2009 National Science Foundation CAREER Award and the 2012 IEEE Control System Technology Award. In 2008, he was appointed the chair of the IEEE technical committee on automotive control.
        \end{IEEEbiography}
\fi

\end{document}